\documentclass[a4paper,11pt]{amsart} 
\usepackage{amsmath, amsfonts, amssymb}
\usepackage[bookmarks=false]{hyperref}
\usepackage[shortlabels]{enumitem}
\usepackage[capitalize]{cleveref}

\usepackage{color}

\numberwithin{equation}{section}

\newtheorem{letterthm}{Theorem}

\newtheorem{letterquestion}[letterthm]{Question}

\newtheorem{lettercor}[letterthm]{Corollary}

\newtheorem{theorem}{Theorem}[section]
\newtheorem{lemma}[theorem]{Lemma}
\newtheorem{corollary}[theorem]{Corollary}
\newtheorem{proposition}[theorem]{Proposition}

\theoremstyle{definition}
\newtheorem{remark}[theorem]{Remark}

\newtheorem{definition}[theorem]{Definition}

\newtheorem*{definition*}{Definition}
\newtheorem*{definitions*}{Definitions}

\newtheorem*{claim*}{Claim}

\newtheorem*{question*}{Question}

\newcommand{\R}{\mathbf{R}}
\newcommand{\C}{\mathbf{C}}
\newcommand{\Z}{\mathbf{Z}}

\newcommand{\Q}{\mathbf{Q}}
\newcommand{\N}{\mathbf{N}}

\newcommand{\T}{\mathbf{T}}

\newcommand{\cC}{\mathcal{C}}

\newcommand{\cU}{\mathcal{U}}

\newcommand{\cR}{\mathcal{R}}

\newcommand{\cP}{\mathcal{P}}

\newcommand{\cb}{\mathfrak{b}}
\newcommand{\cd}{\mathfrak{d}}
\newcommand{\cq}{\mathfrak{q}}

\newcommand{\Ad}{\operatorname{Ad}}
\newcommand{\id}{\text{\rm id}}
\newcommand{\ri}{\text{\rm i}}

\newcommand{\Aut}{\operatorname{Aut}}
\newcommand{\Out}{\operatorname{Out}}
\newcommand{\Inn}{\operatorname{Inn}}
\newcommand{\Bil}{\mathrm{Bil}}
\newcommand{\Quad}{\mathrm{Quad}}
\newcommand{\Bic}{\mathrm{Bic}}
\newcommand{\Lie}{\mathfrak{L}}

\newcommand{\Bint}{\mathfrak{B}_{\rm int}}
\newcommand{\Qint}{\mathfrak{Q}_{\rm int}}
\newcommand{\Lin}{\operatorname{Lin}}
\newcommand{\Hom}{\operatorname{Hom}}
\newcommand{\Ob}{\operatorname{Ob}}

\newcommand{\rL}{\mathord{\text{\rm L}}}
\newcommand{\rC}{\mathord{\text{\rm C}}}

\newcommand{\op}{\mathord{\text{\rm op}}}

\newcommand{\Sd}{\mathord{\text{\rm Sd}}}

\newcommand{\rd}{\mathord{\text{\rm d}}}

\newcommand{\ovt}{\mathbin{\overline{\otimes}}}

\newcommand{\II}{{\rm II}}
\newcommand{\III}{{\rm III}}

\title[Almost almost periodic factors and their 3-cohomology obstructions]{Almost almost periodic type $\III_1$ factors \\ and their 3-cohomology obstructions}

\makeatletter
\@namedef{subjclassname@2020}{%
  \textup{2020} Mathematics Subject Classification}
\makeatother

\begin{document}

\begin{abstract}\noindent
We construct an exemple of a full factor $M$ such that its canonical outer modular flow $\sigma^M : \R \rightarrow \Out(M)$ is almost periodic but $M$ has no almost periodic state. This can only happen if the discrete spectrum of $\sigma^M$ contains a nontrivial integral quadratic relation.  We show how such a nontrivial relation can produce a 3-cohomological obstruction to the existence of an almost periodic state. To obtain our main theorem, we first strengthen a recent result of Bischoff and Karmakar by showing that for any compact connected abelian group $K$, every cohomology class in $ H^3(K,\T)$ can be realized as an obstruction of a $K$-kernel on the hyperfinite $\II_1$ factor. We also prove a positive result : if for a full factor $M$ the outer modular flow $\sigma^M : \R \rightarrow \Out(M)$ is almost periodic, then $M \ovt R$ has an almost periodic state, where $R$ is the hyperfinite $\II_1$ factor. Finally, we prove a positive result for crossed product factors associated to strongly ergodic actions of hyperbolic groups.
\end{abstract}

\author{Amine Marrakchi}
\email{amine.marrakchi@ens-lyon.fr}

\subjclass[2020]{46L10, 46L36, 	20J06, 	22B05}

\keywords{}

\maketitle

\section*{Introduction and statement of the main results}
Let $M$ be a von Neumann algebra\footnote{In this paper, all von Neumann algebras are assumed to have separable predual}. Following \cite{Co72}, we say that a faithful normal state $\varphi$ on $M$ is \emph{almost periodic} if its modular flow $\sigma^\varphi : \R \rightarrow \Aut(M)$ is almost periodic in the sense that $K=\overline{\sigma^\varphi(\R)}$ is a compact subgroup of $\Aut(M)$. Equivalently, $\varphi$ is almost periodic if and only if its modular operator $\Delta_\varphi$ has discrete spectrum. Almost periodic states are important not only because of their greater technical simplicity, but also because they play a key role in the classification of type $\III$ factors. Indeed, in \cite{Co72} shows that an almost periodic state on a type $\III$ factor $M$ produces a decomposition of $M$ as a crossed product $M =N \rtimes \Gamma$ where $N$ is a von Neumann algebra with a semifinite trace and $\Gamma < \R^*_+$ is a \emph{countable} subgroup with a trace-scaling action $\Gamma \curvearrowright N$. Using these so-called \emph{discrete decompositions}, the classification problem reduces to the classification of type $\II$ factors and their outer automorphisms. In \cite{Co75a,Co75b}, Connes applied this approach successfully to all factors of type $\III_\lambda, \: \lambda \in [0,1[$, since these factors always have almost periodic states.

For type $\III_1$ factors, the situation is different. Some of them,  like the unique injective type $\III_1$ factor \cite{Ha85}, do have almost periodic states, but some of them don't. The firt examples of type $\III_1$ factors without almost periodic states were constructed by Connes in \cite{Co74}. Let us briefly explain the key idea of \cite{Co74}. Recall that the modular flow $\sigma^\varphi : \R \rightarrow \Aut(M)$ of a state $\varphi$ does not depend on the choice of $\varphi$ up to inner automorphisms $\Inn(M)$, so that $\sigma^\varphi$ descends to a canonical outer modular flow $\sigma^M : \R \rightarrow \Out(M)=\Aut(M)/\Inn(M)$. Connes' key idea is to consider factors $M$ that are \emph{full}, meaning that $\Inn(M)$ is closed in $\Aut(M)$. For such factors, the quotient group $\Out(M)=\Aut(M)/\Inn(M)$ becomes a Hausdorff topological group and one can look at $\overline{\sigma^M(\R)}$, the closure of $\sigma^M(\R)$ inside $\Out(M)$. If $M$ has an almost periodic state $\varphi$, then $\overline{\sigma^M(\R)}$ must be compact since it is a quotient of $\overline{\sigma^\varphi(\R)}$. But Connes manages to construct examples of full type $\III_1$ factors $M$ for which $\overline{\sigma^M(\R)}$ is prescribed and in particular, examples for which $\overline{\sigma^M(\R)}$ is not compact, thus forbidding the existence of an almost periodic state.

In this paper, we are interested in the following very natural question that arizes from Connes' work.
\begin{letterquestion} \label{main question}
   Suppose that $M$ is a full factor such that $\sigma^M$ is almost periodic, in the sense that the group $K=\overline{\sigma^M(\R)}$ is compact in $\Out(M)$. Does this imply that $M$ has an almost periodic state? 
\end{letterquestion}

Our main result shows that Question \ref{main question} is quite subtle in the sense that the answer depends on the countable subgroup $\Gamma=\widehat{K} < \R^*_+=\widehat{\R}$ obtained by Pontryagin duality from the dense one-parameter subgroup $\sigma^M : \R \rightarrow K=\overline{\sigma^M(\R)}$. Note that this subgroup $\Gamma < \R^*_+$ coincides with the group $\Sd(M)$ defined in \cite{Co74} when $M$ has an almost periodic state. Our main result is the following equivalence which relates a positive answer to Question \ref{main question} to an algebraic condition on $\Gamma$.

\begin{letterthm} \label{main thm diophantine}
    Let $\Gamma < \R^*_+$ be a countable subgroup. The following are equivalent :
    \begin{enumerate}
        \item Every full factor $M$ such that $\sigma^M$ is $\Gamma$-almost periodic has an almost periodic state.
        \item The group $\log(\Gamma) < \R$ is quadratically free : for every $\Z$-linearly independent family $(t_i)_{1 \leq i \leq n}$ in $\log(\Gamma)$, the family $(t_it_j)_{1 \leq i \leq j \leq n}$ is also $\Z$-linearly independent.
    \end{enumerate}
\end{letterthm}
For example, if $\Gamma$ is contained in $\lambda^\Q$ for some $\lambda \in \R^*_+$, then condition (2) is always satisfied. If $\Gamma$ is freely generated by two elements $\lambda, \mu \in \R^*_+$, then condition (2) holds if and only if $\frac{\log(\lambda)}{\log(\mu)}$ is not a quadratic number. However, if we take  $\Gamma=\langle \lambda, \lambda^{\sqrt{2} }\rangle$ for some $\lambda \in \R^*_+ \setminus \{1\}$, then condition (2) fails and we obtain the following corollary. 

\begin{lettercor}
      There exists a full factor $M$ such that $\sigma^M : \R \rightarrow \Out(M)$ is almost periodic but $M$ has no almost periodic state.
\end{lettercor}

Let us now explain the main ideas behind Theorem \ref{main thm diophantine}. Let $M$ be a full factor such that $\sigma^M$ is almost periodic, i.e.\ there exists a compact abelian group $K$, a continuous morphism $\kappa : K \rightarrow \Out(M)$ and a dense one-parameter subgroup $\iota : \R \rightarrow K$ such that $\kappa \circ \iota=\sigma^M$. Then one can show that $M$ has an almost periodic state if and only if the morphism $\kappa$ can be lifted to a continuous morphism $\tilde{\kappa} : K \rightarrow \Aut(M)$ and in that case, we will have $\tilde{\kappa} \circ \iota=\sigma^\varphi$ for some almost periodic weight $\varphi$. So the problem is wether such a lift $\tilde{\kappa}$ exists.

More generally, given a kernel $\kappa : G \rightarrow \Out(M)$ where $M$ is a factor and $G$ is a locally compact group, the problem of wether $\kappa$ can be lifted to a genuine action $\tilde{\kappa} : G \rightarrow \Aut(M)$ is a classical problem that has been studied by several people \cite{Co77, Su80, Jo80}. In \cite[Section 3]{Su80}, Sutherland shows that when $M$ is infinite, the obstruction for the existence of such a lifting action is measured by a 3-cohomology class $\Ob(\kappa) \in H^3(G,\T)$ (see Section \ref{section 3-cocycle}). When $G$ is discrete, Jones proved that every 3-cohomology class $\theta \in H^3(G,\T)$  can be realized as an obstruction $\theta=\Ob(\kappa)$ for some kernel $\kappa : G \rightarrow \Out(M)$ where $M$ is the hyperfinite $\II_1$ factor. However, not much is known when $G$ is a nondiscrete locally compact group, and in particular a compact connected abelian group, which is the case we are interested in.

Very recently, by using a construction from Conformal Field Theory, Bischoff and Karmakar \cite{BK24} proved that if $K=\T^d$ is a finite dimensional torus, then every 3-cohomology class $\theta \in H^3(K,\T)$ can be realized as an obstruction $\theta=\Ob(\kappa)$ for some kernel $\kappa : K \rightarrow \Out(M)$ where $M$ is the injective $\III_1$ factor. We first extend their result to arbitray compact connected abelian groups and we strengthen it by realizing the obstructions on the hyperfinite $\II_1$ factor. We also realize the obstructions on a full $\II_1$ factor. The construction is quite simple and does not involve any conformal net machiney.

\begin{letterthm}
    Let $K$ be a compact connected abelian group. For every cohomology class $\theta \in H^3(K,\T)$, there exists a $\II_1$ factor $M$ and a kernel $\kappa : K \rightarrow \Out(M)$ such that $\Ob(\kappa)=\theta$. Moreover, we can take $M$ to be hyperfinite or we can take it to be full.
\end{letterthm}

Now, in order to obtain Theorem \ref{main thm diophantine}, we need to determine which 3-cohomology classes can be realized as an obstruction of a kernel that comes from an almost periodic outer modular flow $\sigma^M : \R \rightarrow \Out(M)$. In our next main result, we give a precise characterization of these 3-cohomology classes. To formulate our criterion, we need a geometric interpretation of $H^3(K,\T)$ when $K$ is a compact connected abelian group. In section 2, we recall the well-known results about this cohomology group and we show that it is naturally isomorphic to the group $\Qint(K)$ of \emph{integral quadratic forms} on $\Lie(K)$, where $\Lie(K)$ is the Lie algebra of $K$. The elements of $\Lie(K)$ correspond to the infinitesimal generators of one-parameter subgroups of $K$ via the exponential map $\exp_K : \Lie(K) \rightarrow K$. The precise definition of $\Qint(K)$ is a bit technical when $K$ is an arbitrary compact connected abelian group. However, in the special case where $K=\T^n$ is a torus and $\Lie(K)=\R^n$, the group $\Qint(K)$ is nothing more than the set of all quadratic forms on $\R^n$ with integral coefficients.
 
\begin{letterthm} \label{main precise condition}
    Let $K$ be a compact connected abelian group and let $\iota : \R \rightarrow K$ be a dense one-parameter subgroup. Take $\cq \in \Qint(K) \cong H^3(K,\T)$ a tangent integral quadratic form. There exists a factor $M$ and a kernel $\kappa : K \rightarrow \Out(M)$ such that $\Ob(\kappa)=\cq$ and $\kappa \circ \iota = \sigma^M$ if and only if $\cq(\xi)=0$ where $\xi \in \Lie(K)$ is the infinitesimal generator of $\iota$. In that case, we can take $M$ to be hyperfinite of type $\II$ or we can take it to be full (of type $\III$ in general).
\end{letterthm}
Let $K$ be a compact connected abelian group and $\iota : \R \rightarrow K$ a dense one-parameter subgroup. Let $\Gamma=\widehat{K} < \R^*_+=\widehat{\R}$ is the subgroup obtained by Pontryagin duality from $\iota : \R \rightarrow K$. Then an integral quadratic form $\cq \in \Qint(K)$ that satisfies $\cq(\xi)=0$ can be interpreted as a quadratic relation in $\log(\Gamma)$ and this is how we obtain Theorem \ref{main thm diophantine} from Theorem \ref{main precise condition} (see Proposition \ref{quadratic interpretation}).

Another consequence of Theorem \ref{main precise condition} is the following corollary which provides a weak positive solution to Question \ref{main question}.

\begin{lettercor}
      For a full factor $M$, the following properties are equivalent :
    \begin{enumerate}
        \item $\sigma^M : \R \rightarrow \Out(M)$ is almost periodic.
        \item $M \ovt R$ has an almost periodic state, where $R$ is the hyperfinite $\II_1$ factor.
        \item There exists a full factor $N$ such that $M \ovt N$ has an almost periodic state.
        \item There exists a factor $N$ such that $M \ovt N$ has an almost periodic state.
    \end{enumerate}
\end{lettercor}

Finally, we observe that the analog of Question \ref{main question} in the framework of strongly ergodic nonsingular equivalence relation always has a positive answer (Proposition \ref{prop equivalence}). By combining this observation with the results of \cite{HMV17}, we obtain a positive answer to Question \ref{main question} for crossed product factors associated to strongly ergodic actions of hyperbolic groups.

\begin{letterthm} \label{main equivalence relation}
    Let $M= \rL^\infty(X,\mu) \rtimes \Gamma$ where $\Gamma$ is a hyperbolic group (or more generally a bi-exact group) and $\Gamma \curvearrowright (X,\mu)$ is a strongly ergodic nonsingular action. 
    
    Then $\sigma^M : \R \rightarrow \Out(M)$ is almost periodic if and only if $M$ has an almost periodic state if and only if $\Gamma \curvearrowright (X,\mu)$ admits an almost periodic measure $\nu \sim \mu$.
\end{letterthm}

The paper is divided into 6 sections. The first one contains some lengthy preliminaries on compact connected abelian groups and their tangent space and tangent bilinear forms. This whole theory is essentially trivial if we are only interested in applications to tori (compact connected abelian Lie groups). Therefore, we recommend that the reader first focuses on this much simpler case by skipping the technical details of the preliminary section. 

\bigskip

{\bf Acknowledgment.} We thank Cyril Houdayer for our many thought-provoking discussions on Question \ref{main question}.

\break

\tableofcontents

\section{Preliminaries on locally compact abelian groups}
A general reference for this section is \cite[Section 7]{HM20}.

\subsection{Structure of locally compact abelian groups}

\begin{theorem}
Let $G$ be a locally compact abelian group. Then $G$ is isomorphic to a direct product $H \times V$ where $H$ is a locally compact abelian group that has a compact open subgroup and $V \cong \R^n$ is an euclidian group.
\end{theorem}

In particular, if $G$ is a connected locally compact abelian group, then $G$ is isomorphic to $K \times V$ where $K$ is the largest compact and connected and $V$ is euclidian. 

More generally, we say that a locally compact abelian group $G$ is \emph{almost connected} if $G/G_c$ is compact where $G_c$ is the connected component of $G$. In that case, $G$ has a largest compact subgroup $K < G$, the quotient $V=G/K$ is euclidian and the extension $K \rightarrow G \rightarrow V$ is split $G \cong K \times V$. 

Every locally compact abelian group $G$ admits an almost connected open subgroup $G_0 < G$ that is obtained as the preimage of any compact open subgroup of $G/G_c$.

\subsection{Duality for locally compact abelian groups}

Let $G$ and $H$ be two locally compact abelian groups. The set $\Hom(G,H)$ of continuous group morphisms is equipped with the compact-open topology. Then $\Hom(G,H)$ is an abelian topological group. 

The dual group $\widehat{G} { \dot{=}} \Hom(G,\T)$ is a locally compact abelian group. We have a natural isomorphism $G \cong \widehat{\widehat{G}}$ and we will often abusively identify $G$ with $\widehat{\widehat{G}}$. We use the notation $\langle g, \gamma \rangle \in \T$ for the duality pairing between $g \in G$ and $\gamma \in \widehat{G}$. Let $G$ and $H$ be two locally compact abelian groups. Every $\phi \in \Hom(G,H)$ induces a dual map $\widehat{\phi} \in \Hom(\widehat{H},\widehat{G})$ defined by $\langle g, \widehat{\phi}(\gamma) \rangle = \langle \phi(g),\gamma\rangle$ for all $g \in G$ and $\gamma \in \widehat{H}$. The map $\phi \mapsto \widehat{\phi}$ is an isomorphism of topological groups.

If $G$ is a locally compact abelian group and $H < G$ is a closed subgroup, then $H^\perp = \{ \gamma \in \widehat{G} \mid \forall h \in H, \; \langle \gamma, h \rangle= 1 \}$ is a closed subgroup of $\widehat{G}$ and the natural morphism $H^\perp \rightarrow \widehat{G/H}$ is an isomorphism. We have $(H^\perp)^\perp=H$.

Let $G$ be a locally compact abelian group. Then $G$ is compact if and only if $\widehat{G}$ is discrete. If $G_c$ is the connected component of the identity, then $G_c^\perp$ is the union of all compact subgroups of $\widehat{G}$. In particular, $G$ is connected if and only if $\widehat{G}$ has no compact subgroups. Finally, $G$ is compact and connected if and only if $\widehat{G}$ is a torsion free discrete group.

\subsection{The tangent space of a locally compact abelian group}
For the following proposition, we refer to \cite[Proposition 7.36]{HM20}

\begin{proposition} \label{tengent space}
    Let $G$ be a locally compact abelian group. 
    \begin{enumerate}
        \item There exists a pair $(\Lie(G), \exp_G)$ where $\Lie(G)$ is a topological vector space and $\exp_G : \Lie(G) \rightarrow G$ is a continuous group morphism such that $(\Lie(G),\exp_G)$ is universal : for every pair $(V, \varphi )$ where $V$ is a topological vector space and $\varphi : V \rightarrow G$ is a continuous group morphism, there exists a unique continuous linear map $\rd \varphi : V \rightarrow \Lie(G)$ such that $\varphi = \exp_G \circ \: \rd\varphi$.
        \item The map $\Lie(G) \ni \xi \mapsto (t \mapsto \exp_G(t\xi)) \in \Hom(\R,G)$ is a homeomorphism.
        \item The image of $\exp_G : \Lie(G) \rightarrow G$ is dense in the connected component of $G$.
    \end{enumerate}
\end{proposition}

Every continuous morphism $\phi \in \Hom(G,H)$ between two locally compact abelian groups $G$ and $H$ induces a continuous linear map $\rd \phi \in \Hom(\Lie(G), \Lie(H))$ defined by $ \phi \circ \exp_G=\exp_H \circ \: \rd \phi$. The map $\rd : \Hom(G,H) \rightarrow \Hom(\Lie(G),\Lie(H))$ is a continuous group morphism. If $H$ is a closed subgroup of $G$, then $\Lie(H)$ is naturally a closed subspace of $\Lie(G)$.

For the following proposition, see \cite[Theorem 7.66]{HM20}.
\begin{proposition} \label{open differential}
    Let $\pi : G \rightarrow H$ be a morphism between two locally compact abelian groups $G$ and $H$, and let $\rd \pi : \Lie(G) \rightarrow \Lie(H)$ be its tangent map. 
    \begin{enumerate}
        \item We have $\ker(\rd \pi)=\Lie(\ker \pi)$. In particular, if $\ker \pi$ is discrete, then $\rd \pi$ is injective.
        \item If $\pi$ is open, then $\rd \pi$ is also open.
        \item If $\pi$ is a covering map (open with discrete kernel), then $\rd \pi$ is an isomorphism of topological vector spaces.
    \end{enumerate}
\end{proposition}
%\begin{proof} Take $\xi \in \Lie(H)$. Then $t \mapsto \exp(t\xi)$ is a continuous map from $\R$ to $H$. Since $\pi$ is a covering map, there exists a unique continuous lift $\phi : \R \rightarrow G$ such that $\phi(0)=1$ and $\pi(\phi(t))=\exp(t\xi)$ for all $t \in \R$. Put $\varphi(\xi)=\phi(1)$. Then $\varphi : \Lie(H) \rightarrow G$ satisfies $\varphi \circ \rd \pi = \exp_G$ and $\pi \circ \varphi = \exp_H$. In particular, since $\pi$ is an open map, $\varphi$ is a continuous group morphism. Thus there exists a unique $\rd \varphi : \Lie(H) \rightarrow \Lie(G)$ such that $\varphi = \exp_G \circ \: \rd \varphi$. Observe then that $\varphi \circ \rd \pi = \exp_G$ implies $\rd \varphi \circ \rd \pi = \rd(\exp_G)=\id_{\Lie(G)}$ and $\pi \circ \varphi = \exp_H$ implies $\rd \pi \circ \rd \varphi = \rd( \exp_H)=\id_{\Lie(H)}$.\end{proof}

Let $G$ be a locally compact abelian group. A \emph{linear embedding} of $G$ is a weakly continuous group morphism $\theta :G \rightarrow V$ where $V$ is a vector space (not necessarily finite dimensional). Since we do not care about putting a locally convex topology on $V$, we have the following proposition which is a simplified version of \cite[Propositions 7.34 and 7.35]{HM20}.
\begin{proposition} \label{linear enveloppe}
    Let $G$ be a locally compact abelian group.
    \begin{enumerate}
        \item There exists a linear embedding  $\ell_G : G \rightarrow \Lin(G)$ that is universal : for every linear embedding $\theta : G \rightarrow V$ there exists a unique linear map $\Lin(\theta) : \Lin(G) \rightarrow V$ such that $\theta = \Lin(\theta) \circ \ell_G$.
        \item $\Lin(G)^* \ni \phi \mapsto \phi \circ \ell_G \in \Hom(G,\R)$ is a homeomorphism if we put the weak* topology on $\Lin(G)^*$.
        \item $\Lin(G)$ is linearly spanned by the image of $\ell_G$. If $G$ is discrete then $\Lin(G)=\R \otimes G$ with $\ell_G : g \mapsto 1 \otimes g$.
    \end{enumerate}
\end{proposition}

\begin{remark}
If $G$ is an almost connected locally compact abelian group, then $\Lin(G)$ is finite dimensional and $\ker \ell_G$ is the largest compact subgroup of $G$. Moreover, $\ell_G : G \rightarrow \Lin(G)$ is a quotient map and it is the largest euclidian quotient of $G$. In this case, we have $G \cong \ker \ell_G\times \Lin(G)$.
\end{remark}

We identify $\Lie(\T)$ with $\R$, with its exponential map $\exp_\T : t \mapsto e^{2 \ri \pi t}$. Let $G$ be any locally compact abelian group. Since $\widehat{G}=\Hom(G,\T)$, we have a continuous group morphism $$\rd : \widehat{G} \rightarrow \Hom(\Lie(G),\Lie(\T))=\Hom(\Lie(G),\R)=\Lie(G)^*.$$
For every $\gamma \in \widehat{G}$, we call $\rd \gamma \in \Lie(G)^*$ the \emph{tangent linear form} of $\gamma$. It is characterized by the property 
$$\langle \gamma, \exp_G(\xi) \rangle = \exp_\T( \rd\gamma(\xi)) = e^{2 \ri \pi \rd \gamma(\xi) }$$
for all $\xi \in \Lie(G)$.

\begin{proposition} \label{duality tangent space linear}
 Let $G$ be a locally compact abelian group. Let $$\Lin(\rd) : \Lin(\widehat{G}) \rightarrow \Lie(G)^*$$ be the linear map obtained from $\rd : \widehat{G} \rightarrow \Lie(G)^*$. Then  $\Lin(\rd)$ is an isomorphism of topological vector spaces when $\Lin(\widehat{G})$ is equipped with the weak topology and $\Lie(G)^*$ with the weak* topology.
 
 In particular, $\Lie(G)^*$ is linearly spanned by the image of $\rd$.
\end{proposition}
\begin{proof}
By Proposition \ref{tengent space}, Proposition \ref{linear enveloppe} and duality, we have the following chain of isomorphisms of topological vector spaces
$$\Lie(G) \rightarrow \Hom(\R,G) \rightarrow \Hom(\widehat{G},\R) \rightarrow \Lin(\widehat{G})^*$$
where $\Lin(\widehat{G})^*$ is equipped with the weak*-topology. If we write down the isomorphism $\varphi : \Lie(G) \rightarrow \Lin(\widehat{G})^*$, we see that the dual isomorphism $\varphi^* : \Lin(\widehat{G})=\Lin(\widehat{G})^{**} \rightarrow \Lie(G)^*$ is precisely $\Lin(\rd)$.
\end{proof}

\begin{proposition}
    Let $G$ be a locally compact abelian group. For every $\phi \in \Lie(G)^*$, there exists a torus $K$, a map $r \in \Hom(G,K)$ and $\phi_0 \in \Lie(K)^*$ such that $\phi=\phi_0 \circ \rd r$.
\end{proposition}
\begin{proof}
    By Proposition \ref{duality tangent space linear}, we can write $\cb = \sum_{i=1}^n  \lambda_i \rd \gamma_i$ with $(\lambda_i,\gamma_i) \in \R \times \widehat{G} $ for all $i$. Let $K=\T^n$ and for all $i$, let $e_i \in \Hom(K,\T)$ be the $i$th coordinate projection. Let $\phi_0=\sum_{i=1}^n \lambda_i \rd e_i \in \Lie(K)^*$. Let $r=(\gamma_1,\dots,\gamma_n) \in \Hom(G,K)$. Since $\rd \gamma_i=\rd (e_i \circ r)=\rd e_i \circ \rd r$ for all $i$, we see that $\phi = \phi_0 \circ \rd r$.
\end{proof}

\subsection{Tangent bilinear forms}
Let $V,W$ be two topological vector spaces, we denote by $\Bil(V,W)$ the space of all jointly continuous bilinear forms on $V \times W$. If $V=W$ we denote it simply by $\Bil(V)$. 

For the following proposition, we refer to \cite[Proposition A.7.20]{HM20}.
\begin{proposition} \label{bilinear is algebraic tensor product}
    Let $G$ and $H$ be locally compact abelian groups.  The natural linear map $$ \Lie(G)^* \otimes_\R \Lie(H)^* \rightarrow \Bil(\Lie(G),\Lie(H))$$ is an isomorphism of vector spaces, where $\otimes$ is the algebraic tensor product of vector spaces.

    In particular, $\Bil(\Lie(G),\Lie(H))$ is linearly spanned by the set
    $ \{ \rd \gamma \otimes \rd \delta \mid (\gamma,\delta) \in \widehat{G} \times \widehat{H} \}$.
\end{proposition}

\begin{proposition} \label{bilinear with torus}
    Let $G$ and $H$ be locally compact abelian groups. Take $\mathfrak{b} \in \Bil(\Lie(G),\Lie(H))$. There exist two tori $K$ and $L$, two maps $r \in \Hom(G,K)$ and $s \in \Hom(H,L)$ and $\mathfrak{b}_0 \in \Bil(\Lie(K),\Lie(L))$ such that $\mathfrak{b}=\mathfrak{b}_0 \circ (\rd r \times \rd s)$. 
 
 If $G=H$, we may take $H=L$ and $r=s$.
\end{proposition}
\begin{proof}
    Suppose that $\cb \in \Bil(\Lie(G),\Lie(H))$.  By Proposition \ref{bilinear is algebraic tensor product}, we can write $\cb = \sum_{i=1}^n \lambda_i \rd \gamma_i \otimes \rd \delta_i$ with $(\lambda_i,\gamma_i,\delta_i) \in \R \times \widehat{G} \times \widehat{H}$ for all $i$. Let $K=\T^n$ and for all $i$, let $\omega_i \in \widehat{K}$ be the $i$th coordinate projection. Let $\cb_0=\sum_{i=1}^n \lambda_i \rd \omega_i \otimes \rd \omega_i \in \Bil(\Lie(K))$. Let $r=(\gamma_1,\dots,\gamma_n) \in \Hom(G,K)$ and $s=(\delta_1,\dots,\delta_n) \in \Hom(H,K)$. Since $\rd \gamma_i=\rd \omega_i \circ \rd r$ and $\rd \delta_i= \rd \omega_i \circ \rd s$ for all $i$, we see that $\cb = \cb_0 \circ (\rd r \times \rd s)$.

    If $G=H$, we take $K=\T^{2n}$, $\cb_0=\sum_{i=1}^n \lambda_i \rd \omega_i \otimes \rd \omega_{i+n}$ and $r=s=(\gamma_1,\dots,\gamma_n,\delta_1,\dots,\delta_n) \in \Hom(G,K)$.
\end{proof}

For a topological vector space $V$, we denote by $\Quad(V)$ the set space of all continuous quadratic forms on $V$, which is the quotient of $\Bil(V)$ by the closed subspace of all anti-symmetric forms. Then Proposition \ref{bilinear is algebraic tensor product} implies the following analog for quadratic forms.

\begin{proposition} \label{quadratic is algebraic symmetric square}
    Let $G$ be a locally compact abelian group.  The natural linear map $$ \Lie(G)^* \odot \Lie(G)^* \rightarrow \Quad(\Lie(G))$$ is an isomorphism of vector spaces, where $\odot$ is the symmetric tensor product.

In particular, $\Quad(\Lie(G))$ is linearly spanned by the set
    $ \{ \rd \gamma \odot \rd \delta \mid (\gamma,\delta) \in \widehat{G} \times \widehat{G} \}$.
\end{proposition}

Proposition \ref{bilinear with torus} implies the following analog for quadratic forms.
\begin{proposition}\label{quadratic with torus}
    Let $G$ be a locally compact abelian group. Take $\cq \in \Quad(\Lie(G))$. There exist a torus $K$, some $r \in \Hom(G,K)$ and $\cq_0 \in \Quad(\Lie(K))$ such that $\cq=\cq_0 \circ \rd r$.
\end{proposition}

\subsection{Integral bilinear forms}
\begin{proposition} 
    Let $K$ and $L$ be compact connected abelian groups. Then the natural map $$\rd \otimes \rd  : \widehat{K} \otimes \widehat{L} \ni \gamma \otimes \delta \mapsto \rd \gamma \otimes \rd \delta \in \Bil(\Lie(K),\Lie(L))$$
    is injective and extends to a linear isomorphism 
    $$ \R \otimes \widehat{K} \otimes \widehat{L} \rightarrow \Bil(\Lie(K),\Lie(L)).$$
\end{proposition}
\begin{proof}
Since $\widehat{K}$ and $\widehat{L}$ are torsion free discrete groups, the natural embeddings maps $\ell_K : \widehat{K} \rightarrow \Lin(\widehat{K})=\R \otimes \widehat{K}$ and $\widehat{L} \rightarrow \Lin(\widehat{L})=\R \otimes \widehat{L}$ are injective. Then Propositions \ref{duality tangent space linear} and \ref{bilinear is algebraic tensor product} show that $\rd \otimes \rd$ extends to a linear isomorphism from
$$ \Lin(K) \otimes_\R \Lin(L) = (\R \otimes \widehat{K}) \otimes_\R (\R \otimes \widehat{L})= \R \otimes \widehat{K} \otimes \widehat{L}$$
onto  $\Bil(\Lie(K),\Lie(L))$. 
\end{proof}

\begin{definition}
Let $K$ and $L$ be compact connected abelian groups. 
\begin{itemize}
    \item We say that $\cb \in \Bil(\Lie(K),\Lie(L))$ is \emph{integral} if it is in the image of the map $\rd \otimes \rd :  \widehat{K} \otimes \widehat{L} \rightarrow \Bil(\Lie(K),\Lie(L))$.
    \item We denote by $\Bint(K,L)$ the set of all integral bilinear forms in $\Bil(\Lie(K),\Lie(L))$. By definition, we have a group isomorphism $\rd \otimes \rd  : \widehat{K} \otimes \widehat{L} \rightarrow \Bint(K,L)$.
\end{itemize}
\end{definition}
The terminology \emph{integral} is justified by the following proposition.
\begin{proposition}
Let $K$ and $L$ be compact connected abelian groups. Take $\cb \in \Bil(\Lie(K),\Lie(L))$ and consider the following properties.
\begin{enumerate}
    \item $\cb \in \Bint(K,L)$.
    \item $\cb(\xi,\eta) \in \Z$ for all $\xi \in \ker \exp_K$ and $\eta \in \ker \exp_L$.
\end{enumerate}
Then $(1) \Rightarrow (2)$. If $K$ and $L$ are tori, we also have $(2) \Rightarrow (1)$.
\end{proposition}
\begin{proof}
$(1) \Rightarrow (2)$. Take $\gamma \in \widehat{K}$ and $\xi \in \ker \exp_K$. Since $\exp_T \circ \: \rd \gamma=\gamma \circ \exp_K$, we have $(\exp_T \circ \: \rd \gamma)(\xi)=1$, hence $\rd \gamma(\xi) \in \Z$. Similarly, if $\delta \in \widehat{L}$ and $\eta \in \ker \exp_L$, then $\rd \delta(\eta) \in \Z$. Thus $$(\rd \gamma \otimes \: \rd \delta)(\xi,\eta)=\rd \gamma(\xi) \cdot \rd \delta(\eta) \in \Z.$$

$(2) \Rightarrow (1)$. We assume that $K$ and $L$ are tori. Then $\Gamma=\ker \exp_K$ and $\Lambda=\ker \exp_L$ are lattices in $\Lie(K)$ and $\Lie(L)$ respectively. Let $E$ be a basis of $\Gamma$ and $F$ a basis of $\Lambda$. Let $E^*=\{ e^* \mid e \in E\} \subset \Lie(K)^*$ and $F^*=\{ f^* \mid f \in F\} \subset \Lie(L)^*$ be their dual basis. Then we have 
$$ \cb= \sum_{(e,f) \in E \times F} \cb(e,f) \: e^* \otimes f^*$$
where $\cb(e,f) \in \Z$ for all $(e,f) \in E \times F$, by assumption.

For every $e \in E$, the character $\exp_\T \circ \: e^*$ vanishes on $\Gamma=\ker \exp_K$, so it is of the form $\exp_\T \circ \: e^*=\gamma_e \circ \exp_K$ for some $\gamma_e \in \widehat{K}$. This means that $e^*=\rd \gamma_e$. Similarly, for all $f \in F$, we can write $f^*=\rd \delta_f$ for some $\delta_f \in \widehat{L}$. We conclude that 
$$ \cb= \sum_{(e,f) \in E \times F} \cb(e,f) \: \rd \gamma_e \otimes \rd \delta_f \in \Bint(K,L).$$
\end{proof}

\begin{proposition} \label{integral bilinear with torus}
Let $K$ and $L$ be compact connected abelian groups. Take $\cb \in \Bil(\Lie(K),\Lie(L))$. Then $\cb \in \Bint(K,L)$ if and only if there exists two tori $R$ and $S$, two maps $r \in \Hom(K,R)$ and $s \in \Hom(L,S)$ and $\cb_0 \in \Bint(R,S)$ such that $\cb=\cb_0 \circ (\rd r \times \rd s)$.

If $K=L$, we may take $R=S$ and $r=s$.
\end{proposition}
\begin{proof}
    Suppose that $\cb \in \Bint(K,L)$. Write $\cb = \sum_{i=1}^n \rd \gamma_i \otimes \rd \delta_i$ with $(\gamma_i,\delta_i) \in \widehat{K} \times \widehat{L}$ for all $i$. Let $R=\T^n$ and for all $i$, let $\omega_i \in \widehat{R}$ be the $i$th coordinate projection. Let $\cb_0=\sum_{i=1}^n \rd \omega_i \otimes \rd \omega_i \in \Bint(R,R)$. Let $r=(\gamma_1,\dots,\gamma_n) \in \Hom(K,R)$ and $s=(\delta_1,\dots,\delta_n) \in \Hom(L,R)$. Since $\rd \gamma_i=\rd \omega_i \circ \rd r$ and $\rd \delta_i= \rd \omega_i \circ \rd s$ for all $i$, we see that $\cb = \cb_0 \circ (\rd r \times \rd s)$.

    If $K=L$, we take $R=\T^{2n}$, $\cb_0=\sum_{i=1}^n \rd \omega_i \otimes \rd \omega_{i+n}$ and $r=s=(\gamma_1,\dots,\gamma_n,\delta_1,\dots,\delta_n) \in \Hom(K,R)$.
\end{proof}

\begin{proposition}
Let $K$ be a compact connected abelian group. Take $\cq \in \Quad(\Lie(K))$ and consider the following properties.
\begin{enumerate}
    \item There exists $\cb \in \Bint(K,K)$ such that $\cq(\xi)=\cb(\xi,\xi)$ for all $\xi \in \Lie(K)$.
    \item $\cq(\xi) \in \Z$ for all $\xi \in \ker \exp_K$.
\end{enumerate}
Then $(1) \Rightarrow (2)$. If $K$ is a torus, then $(2) \Rightarrow (1)$.
\end{proposition}
\begin{proof}
    $(1) \Rightarrow (2)$ is obvious. Let us prove that $(2) \Rightarrow (1)$ when $K$ is a torus. Define $\cb \in \Bil(\Lie(K),\Lie(K))$ by $\cb(\xi,\eta)=\frac{1}{2}( \cq(\xi+\eta) - \cq(\xi)-\cq(\eta))$ for all $\xi, \eta \in \Lie(K)$. Observe that $2\cb \in \Bint(K,K)$ but in general $\cb \notin \Bint(K,K)$. So we have to modify $\cb$. Let $\Gamma =\ker \exp_K$ and let $(e_1,\dots,e_n)$ be a basis of $\Gamma$. Then
    $$\cb = \sum_{i,j \leq n} \cb(e_i,e_j) e_i^* \otimes e_j^*.$$
    Define $\cb' \in \Bil(\Lie(K),\Lie(K))$ by the formula
    $$ \cb'=\sum_{i=1}^n \cb(e_i,e_i) e_i^* \otimes e_i^* +  \sum_{i < j \leq n} 2\cb(e_i,e_j) e_i^* \otimes e_j^*.$$
    Then $\cb'$ is integral because $\cb(e_i,e_i) =\cq(e_i) \in \Z$ for all $i$ and $2\cb(e_i,e_j)=\cq(e_i+e_j)-\cq(e_i)-\cq(e_j) \in \Z$ for all $i < j \leq n$. Moreover, we have $\cb'(\xi,\xi)=\cb(\xi,\xi)=\cq(\xi)$ for all $\xi \in \Lie(K)$ as we wanted.
\end{proof}

\begin{definition}
Let $K$ be a compact connected abelian group. 
\begin{itemize}
    \item We say that $\cq \in \Quad(\Lie(K))$ is \emph{integral} if it satisfies condition (1) above.
    \item We denote by $\Qint(K)$ the set of all integral quadratic forms on $\Lie(K)$.
\end{itemize}
\end{definition}

\begin{proposition} \label{integral quadratic with torus}
Let $K$ be a compact connected abelian group. Take $\cq \in \Quad(\Lie(K))$. Then $\cq \in \Qint(K)$ if and only if there exists a torus $R$, a map $r \in \Hom(K,R)$ and $\cq_0 \in \Qint(R)$ such that $\cq=\cq_0 \circ \rd r $. 
\end{proposition}
\begin{proof}
Follows directly from Proposition \ref{integral bilinear with torus} and the definition $\Qint(K)$.
\end{proof}

\begin{proposition} \label{anti-symmetric is difference}
Let $K$ be a compact connected abelian group. Take $\cb \in \Bint(K)$. Then $\cb$ is anti-symmetric if and only if there exists $\mathfrak{d} \in \Bint(K)$ such that $\cb=\mathfrak{d}- \mathfrak{d}^{\op}$.
\end{proposition}
\begin{proof}
Thanks to Proposition \ref{integral bilinear with torus}, it is enough to prove the proposition when $K$ is a torus. Let $\Gamma =\ker \exp_K$ and let $(e_1,\dots,e_n)$ be a basis of $\Gamma$. Then
    $$\cb = \sum_{1 \leq i,j \leq n} \cb(e_i,e_j) e_i^* \otimes e_j^*.$$
    Since $\cb$ is integral anti-symmetric, we have $\cb(e_i,e_j) \in \Z$ and $\cb(e_i,e_j)=-\cb(e_j,e_i)$ for all $i,j$. Thus, we can take $\cd \in \Bint(K)$ defined by the formula
    $$ \cd=\sum_{i < j} \cb(e_i,e_j) e_i^* \otimes e_j^*.$$
\end{proof}

If $A$ is an abelian group, we denote by $A \odot A$ the \emph{symmetric} tensor square of $A$, defined as the quotient of $A \otimes A$ by the subgroup $\{ x-x^{\op} \mid x \in A \otimes A\}$.

\begin{proposition} \label{isomorphism square quadratic}
Let $K$ be a compact connected abelian group. Then the group isomorphism
$$\rd \otimes \rd  : \widehat{K} \otimes \widehat{K} \ni \gamma \otimes \delta \mapsto \rd \gamma \otimes \rd \delta \in \Bint(K)$$
 descends to a group isomorphism
$$ \rd \odot \rd  : \widehat{K} \odot \widehat{K} \ni \gamma \odot\delta \mapsto \rd \gamma \odot \rd \delta \in \Qint(K)$$
\end{proposition}
\begin{proof}
Let $p : \widehat{K} \otimes \widehat{K} \rightarrow \widehat{K} \odot \widehat{K}$ be the canonical projection. Let $\pi : \Bint(K) \rightarrow \Qint(K)$ be the projection that sends a bilinear form $\cb$ to its associated quadratic form $\xi \mapsto \cb(\xi,\xi)$. Then $\ker \pi$ consists of anti-symmetric bilinear maps and $(\rd \otimes \rd)(\ker p)=\{ \cd - \cd^{\op} \mid \cd \in \Bint(K) \}$. By Proposition \ref{anti-symmetric is difference}, these two groups are equal and the conclusion follows.
\end{proof}

\subsection{Covering bilinear forms}
\begin{definition}
    Let $G$ and $H$ be locally compact abelian groups. 
    \begin{itemize}
        \item A \emph{bilinear form} on $(G,H)$ is a continuous bimorphism from $(G,H)$ to $\R$, i.e.\ a continuous map $B : G \times H \rightarrow \R$ such that $B(g,\cdot) \in \Hom(H,\R)$ and $B(\cdot, h) \in \Hom(G,\R)$ for all $(g,h) \in G \times H$.
        \item  We denote by $\Bil(G,H)$ the space of all bilinear forms on $(G,H)$. If $G=H$, we denote it simply $\Bil(G)$.
        \item For $B \in \Bil(G,H)$, its \emph{tangent bilinear form} $\rd B \in \Bil(\Lie(G),\Lie(H))$ is defined by $\rd B= B \circ (\exp_G \times \exp_H)$.
    \end{itemize}
\end{definition}

\begin{proposition} \label{bilinear locally vanish}
    Let $G$ and $H$ be locally compact abelian groups. Take $B \in \Bil(G,H)$. Then $\rd B=0$ if and only if there exist open subgroups $G_0 < G$ and $H_0 < H$ such that $B|_{G_0 \times H_0}=0$.
\end{proposition}
\begin{proof} Let $G_c$ and $H_c$ be the connected components of $G$ and $H$ respectively. The condition $\rd B =0$ means that $B|_{G_c \times H_c}=0$. Let $G_0$ and $H_0$ be any open subgroups of $G$ and $H$ respectively such that $G_0/G_c$ and $H_0/H_c$ are compact. Since $B|_{G_c \times H_c}=0$ is trivial, we can view $B|_{G_c \times H_0}$ as an element of $\Bil(G_c,H_0/H_c)$. Thus $B|_{G_c \times H_0}=0$ because $H_0/H_c$ is compact. But this means that we can view $B|_{G_0 \times H_0}$ as an element of $\Bil(G_0/G_c,H_0)$. Since $G_0/G_c$ is compact again, we conclude that $B|_{G_0 \times H_0}=0$.
\end{proof}

\begin{proposition} \label{locally anti-symmetric}
    Let $G$ be a locally compact abelian group. Take $B \in \Bil(G)$. Then $\rd B$ is symmetric (resp.\ anti-symmetric) if and only if there exists an open subgroup $G_0 < G$ such that $B|_{G_0 \times G_0}$ is symmetric (resp.\ anti-symmetric).
\end{proposition}
\begin{proof}
Apply Proposition \ref{bilinear locally vanish} to $B-B^{\op}$ or $B+B^{\op}$.
\end{proof}

\begin{definition}
    Let $G$ be a locally compact abelian group. 
    \begin{itemize}
    		\item A \emph{covering} of $G$ is a morphism $p \in \Hom(G',G)$ where $G'$ is locally compact abelian group such that $p$ is open and $\ker p$ is discrete in $G'$. 
    		\item If $p_i : G'_i \rightarrow G,\: i=1,2$ are two coverings of $G$, a morphism $r \in \Hom(p_1,p_2)$ is a map $r \in \Hom(G'_1,G'_2)$ such that $p_2 \circ r=p_1$. 
        \item A \emph{covering bilinear form} on $G$ is a pair $(p,B)$ where $p : G' \rightarrow G$ is a covering and $B \in \Bil(G')$.  We denote by $\Bil^\downarrow(G)$ the set of all covering bilinear forms on $G$.
        \item A morphism between two covering bilinear forms $(p_1,B_1), (p_2,B_2) \in \Bil^\downarrow(G,H)$ is a map $r \in \Hom(p_1,p_2)$ such that $B_1 = B_2 \circ r$. If such a morphism $r$ exists, we say that $(p_1,B_1)$ covers $(p_2,B_2)$, or that $(p_2,B_2)$ is covered by $(p_1,B_1)$.
        \item We say that $(p_1,B_1)$ and $(p_2,B_2)$ in $\Bil^\downarrow(G)$  are \emph{locally equivalent} if they can be covered by the same covering bilinear form $(p_3,B_3)$.
        \item If $(p,B) \in \Bil^\downarrow(G)$ where $p : G' \rightarrow G$ is a covering morphism, then the tangent bilinear form $\rd B \in \Bil(\Lie(G'))$ is abusively identified with $\rd B \circ (\rd p \times \rd p)^{-1} \in \Bil(\Lie(G))$. Recall that $\rd p$ is an isomorphism by Proposition \ref{open differential}.
    \end{itemize}    
\end{definition}

\begin{remark} \label{remark covering form on torus}
    Let $K$ be a torus and consider its universal covering $\exp_K : \Lie(K) \rightarrow K$. Then for all $\cb \in \Bint(K)$, the pair $(\exp_K,\cb)$ is a covering bilinear form and $\rd \cb=\cb$. So this whole section is trivial and can be skipped when $K$ is a torus.
\end{remark}

\begin{remark} \label{remark disconnected covering}    
   On the other hand if $K$ is an arbitrary compact connected abelian group, $K$ does not have a universal covering group in general. Moreover, if $K$ is not locally connected (for example if $K$ is the dual group of $\Q$), then the covering groups of $K$ are also not locally connected and their connected component is not open in general. Thus we are forced to work with disconnected coverings. However since any locally compact abelian group admits an open subgroup that is almost connected and any open subgroup of a covering group is again a covering group.

Note that we do not require coverings to be surjective in general but this is automatic if the covered group is connected.
\end{remark}

The main motivation for the lengthy definitions above is the following one-to-one correspondence between tangent bilinear forms and local equivalence classes of covering bilinear forms.
\begin{proposition} \label{local equivalence and tangent form}
    Let $G$ be a locally compact abelian group. 
    \begin{enumerate}
        \item For every $\mathfrak{b} \in \Bil(\Lie(G))$, there exists $(p,B) \in \Bil^\downarrow (G)$ such that $\rd B =\mathfrak{b}$.
        \item  $(p_1,B_1),(p_2,B_2) \in \Bil^\downarrow (G)$ are locally equivalent if and only if $\rd B_1 = \rd B_2$.
    \end{enumerate}
\end{proposition}
\begin{proof} 
    (1) There exists $r \in \Hom(G ,K)$ where $K$ is a torus and $\mathfrak{b}_0 \in \Bil(\Lie(K))$ such that $\mathfrak{b} = \mathfrak{b}_0 \circ (\rd r \times \rd r)$. Let $p \in \Hom(r \times_K \exp_K, r)$ be the first coordinate projection and let $r' \in \Hom(r \times_K \exp_K, \exp_K)$ be the second coordinate projection. Then $p$ is a covering of $G$ and $\exp_K \circ  \: r' = r \circ p$. Define the bilinear form $B = \cb_0 \circ (r' \times r')$. Then $(p,B) \in \Bil^\downarrow(G)$ and by construction, we have $$\rd B = \cb_0 \circ (\rd r' \times \rd r')=\cb_0 \circ (\rd r \times \rd r)=\cb.$$
    
    (2) The only if direction is obvious. For the if direction, suppose that $\rd B = \rd D$. Let $p=p_1 \times_G p_2 : G' \rightarrow G$ be the fiber product of $p_1$ and $p_2$ and let $r_i \in \hom(p,p_i)$ be the coordinate projections. Let $\tilde{B}_i=B_i \circ (r_i \times r_i)$. We then have 
    $$ \rd( \tilde{B}_1 - \tilde{B}_2)=\rd \tilde{B}_1 - \rd \tilde{B}_2 = \rd B_1- \rd B_2=0.$$
    By Lemma \ref{bilinear locally vanish}, $\tilde{B}_1-\tilde{B}_2$ must vanish on $G'_0 \times G'_0$ for some open subgroup $G_0' < G'$. We conclude that $C=\tilde{B}_1|_{G'_0 \times G'_0}=\tilde{B}_2|_{G'_0 \times G'_0}$ covers both of $B_1$ and $B_2$, which means that $B_1$ and $B_2$ are locally equivalent.
\end{proof}

\begin{definition}
    Let $K$ be a compact connected abelian group. 
    \begin{itemize}
        \item We say that $(p,B) \in \Bil^\downarrow(K)$ is \emph{integral} if $B(\ker p \times \ker p) \subset \Z$.
        \item We denote by $\Bil_{\rm int}^\downarrow(K)$ the set of all integral covering bilinear forms on $K$.
    \end{itemize} 
\end{definition}

\begin{proposition} \label{integral covering of integral bilinear form}
    Let $K$ be a compact connected abelian group. Take $\cb \in \Bil(\Lie(K))$. Then $\cb \in \Bint(K)$ if and only if there exists $(p,B) \in \Bil_{\rm int}^\downarrow (K)$ such that $\rd B=\cb$.
\end{proposition}
\begin{proof}
     Suppose that there exists $(p,B) \in \Bil_{\rm int}^\downarrow (K)$ such that $\rd B=\cb$. Let $p : G \rightarrow K$ be the covering group. Up to replacing the group $G$ by some open subgroup, we may assume that $G$ is almost connected. Then $\Lin(G)$ is a finite dimensional vector space and the map $\ell_{G} : G \rightarrow \Lin(G)$ is proper. Let $\Gamma=\iota_{G}(\ker p)$. Since $\iota_{G}$ is proper and $\ker p$ is a cocompact lattice in $G$, then $\Gamma$ is also a lattice in $\Lin(G)$. Let $R=\Lin(G)/\Gamma$ be the quotient group, which is a torus. Make the identification $\Lie(R)=\Lin(G)$ and let $\exp_{R} : \Lin(G) \rightarrow R$ be the quotient map. By definition of $\Gamma$, the map $\exp_{R} \circ \: \ell_{G}$ vanishes on $\ker p$. Therefore, there exists a unique $r \in \Hom(K,R)$ such that $\exp_{R} \circ \: \ell_{G}=r \circ p$. We then have $\rd r =\ell_{G} \circ \exp_{G}$.

    By applying the universal property of $\ell_{G}$, we can write $B=\cb_0 \circ (\ell_{G} \times \ell_{G})$ for some $\cb_0 \in \Bil(\Lin(G))$. We can now compute :
    $$ \cb = \rd B=B \circ (\exp_{G} \times \exp_{G}) = \cb_0 \circ (\ell_{G} \times \ell_{G}) \circ (\exp_{G} \times \exp_{G}) = \cb_0 \circ (\rd r \times \rd r) .$$
    Since $B$ is integral on $\ker p \times \ker p$, then $\cb_0$ is also integral on $\Gamma \times \Gamma$. This means that $\cb_0 \in \Bint(R)$ and we conclude that $\cb \in \Bint(K)$.
\end{proof}

\begin{definition}
    Let $G$ be a locally compact abelian group.
    \begin{itemize}
        \item A \emph{quadratic form} on $G$ is a continuous map $Q : G \rightarrow \R$ such that $(g,h) \mapsto Q(gh)-Q(g)-Q(h)$ is a bilinear form. 
        \item We denote by $\Quad(G)$ the set of all quadratic forms on $G$.
        \item For $Q \in \Quad(G)$, its \emph{tangent quadratic form} $\rd Q \in \Quad(\Lie(G)$ is defined by $\rd Q= Q \circ \exp_G$.
        \item A \emph{covering quadratic form} on $G$ is a pair $(p,Q)$ where $p : G' \rightarrow G$ is a covering morphism and $Q \in \Quad(G')$.  We denote by $\Quad^\downarrow(G)$ the set of all covering quadratic forms on $G$.
        \item A morphism between two covering quadratic forms $(p_1,Q_1), (p_2,Q_2) \in \Bil^\downarrow(G,H)$ is a morphism $r \in \Hom(p_1,p_2)$ such that $Q_1 = Q_2 \circ r $. If such a morphism $r$ exists, we say that $(p_1,Q_1)$ covers $(p_2,Q_2)$, or that $(p_2,Q_2)$ is covered by $(p_1,Q_1)$.
        \item We say that $(p_1,Q_1)$ and $(p_2,Q_2)$ are \emph{locally equivalent} if they can be covered by the same quadratic form $(p_3,Q_3)$.
        \item If $(p,Q) \in \Quad^\downarrow(G)$ where $p : G' \rightarrow G$ is a covering morphism, then the tangent quadratic form $\rd Q \in \Bil(\Lie(G'))$ is abusively identified with $\rd Q \circ (\rd p)^{-1} \in \Quad(\Lie(G))$.
     \end{itemize}
\end{definition}

\begin{proposition}
    Let $G$ be a locally compact abelian group. 
    \begin{enumerate}
        \item For every $\cq \in \Quad(\Lie(G))$, there exists $Q \in \Quad^\downarrow (G)$ such that $\rd Q =\cq$.
        \item  $Q_1,Q_2 \in \Quad^\downarrow (G)$ are locally equivalent if and only if $\rd Q_1 = \rd Q_2$.
    \end{enumerate}
\end{proposition}

\begin{definition}
    Let $K$ be a compact connected abelian group. 
    \begin{itemize}
        \item We say that $(p,Q) \in \Quad^\downarrow(K)$ is \emph{integral} if $Q(\ker p) \subset \Z$.
        \item We denote by $\Quad_{\rm int}^\downarrow(K)$ the set of all integral covering quadratic forms on $K$.
    \end{itemize} 
\end{definition}

\begin{proposition} \label{integral covering quadratic form}
    Let $K$ be a compact connected abelian group.
    Take $\cq \in \Quad(\Lie(K))$. Then $\cq \in \Qint(K)$ if and only if there exists $(p,Q) \in \Quad_{\rm int}^\downarrow(K)$ such that $\rd Q=\cq$.
\end{proposition}

\subsection{Bicharacters}

\begin{definition}
Let $G$ and $H$ be two locally compact abelian groups. A \emph{bicharacter} on $(G,H)$ is a continuous bimorphism from $(G,H)$ to $\T$.  We denote by $\Bic(G,H)$ the set of all bicharacters on $(G,H)$. 
\end{definition}
A bicharacter in $\Bic(G,H)$ can also be viewed as an element of $\Hom(G, \widehat{H})$ (or $\Hom(H, \widehat{G})$).

If $B \in \Bil(G,H)$, then $\chi=\exp_\T \circ B \in \Bic(G,H)$. Not all bicharacters can be realized this way. However, this can be done \emph{locally}, in the following sense. 

\begin{proposition} \label{tangent of bicharacter}
    Let $G$ and $H$ be two locally compact abelian groups. For every $\chi \in \Bic(G,H)$, there exists a unique bilinear form $\rd \chi \in \Bil(\Lie(G),\Lie(H))$ (called the tangent bilinear form of $\chi$) such that $$\chi \circ  (\exp_G \times \exp_H)=\exp_\T \circ \:  \rd \chi.$$

    Moreover, if $\chi=\exp_\T \circ B$ for some $B \in \Bil(G,H)$, then $\rd \chi = \rd B$.
\end{proposition}
\begin{proof}
For every $(\xi,\eta) \in \Lie(G) \times \Lie(G)$, the map $(t,s) \mapsto \chi (\exp_G(t\xi), \exp_H(t\eta))$ is a bicharacter on $\R$. Therefore it is of the form $(t,s) \mapsto \exp_T(ts \cb(\xi,\eta))$ for a unique $\cb(\xi,\eta) \in \R$. It is straightforward to check that the map $\cb : (\xi,\eta) \mapsto \cb(\xi,\eta)$ is in $\Bil(\Lie(G),\Lie(H))$ and satisfies the desired properties.
\end{proof}

\begin{lemma} \label{lem bic connected compact}
    Let $G,H$ be two locally compact abelian groups and $\chi \in \Bic(G,H)$. If $H$ is compact, then there exists an open subgroup $G_0 < G$ such that $\chi|_{G_0 \times H}$ is trivial. In particular, if $G$ is connected and $H$ is compact, then $\chi$ is trivial.
\end{lemma}
\begin{proof}
View $\chi$ as an element of $\Hom(G,\widehat{H})$. Since $H$ is compact, $\widehat{H}$ is discrete. Thus $\chi$ is trivial on some open subgroup of $G$.
\end{proof}

\begin{proposition} \label{bicharacter locally trivial}
    Let $G,H$ be two locally compact abelian groups. Then $\chi \in \Bic(G,H)$ satisfies $\rd \chi = 0$ if and only if there exist two open subgroups $G_0 < G$ and $H_0 \subset H$ such that $\chi|_{G_0 \times H_0}$ is trivial.
\end{proposition}
\begin{proof}
    Let $G_c$ and $H_c$ be the connected components of $G$ and $H$ respectively. The condition $\rd \chi =0$ means that $\chi|_{G_c \times H_c}$ is trivial. Let $G_0$ and $H_0$ be any open subgroups of $G$ and $H$ respectively such that $G_0/G_c$ and $H_0/H_c$ are compact. Since $\chi|_{G_c \times H_c}$ is trivial, we can view $\chi|_{G_c \times H_0}$ as an element of $\Bic(G_c,H_0/H_c)$. Thus $\chi|_{G_c \times H_0}$ is trivial by Lemma \ref{lem bic connected compact}. But this means that we can view $\chi|_{G_0 \times H_0}$ as an element of $\Bic(G_0/G_c,H_0)$. Since $G_0/G_c$ is compact, by applying Lemma \ref{lem bic connected compact} again, we can find some open subgroup $H_1 \subset H_0$ such that $\chi|_{G_0 \times H_1}$ is trivial. This is what we wanted.
\end{proof}

A bicharacter $\chi \in \Bic(G)$ is \emph{alternating} if $\chi(g,g)=1$ for all $g \in G$. This implies that $\chi$ anti-symmetric and the converse is also true up to passing to and index $2$ subgroup.
\begin{proposition} \label{bicharacter alternating}
    Let $G$ be a locally compact abelian groups. Take $\chi \in \Bic(G)$. Then $\rd \chi$ is symmetric (resp.\ anti-symmetric) if and only if there exists an open subgroup $G_0 < G$ such that $\chi|_{G_0 \times G_0}$ is symmetric (resp.\ alternating).
\end{proposition}
\begin{proof}
	Let $\chi^{\rm op}$ be the opposite character of $\chi$. If $\rd \chi$ is symmetric, we apply Proposition \ref{bicharacter locally trivial} to $\psi = \chi \cdot (\chi^{\rm op})^{-1}$. Then $\rd \psi = \rd \chi - (\rd \chi)^{\rm op}=0$. Then there exists an open subgroup $G_0 < G$ such that $\psi|_{G_0 \times G_0}$ is trivial. This means that $\chi|_{G_0 \times G_0}$ is symmetric. 
	
    If $\rd \chi$ is anti-symmetric, we apply Proposition \ref{bicharacter locally trivial} to $\psi = \chi \cdot \chi^{\rm op}$. Then $\rd \psi = \rd \chi + (\rd \chi)^{\rm op}=0$. Thus, there exists an open subgroup $G_0 \subset G$ such that $\psi|_{G_0 \times G_0}$ is trivial. This means that $\chi^{\rm op}|_{G_0 \times G_0}$ is anti-symmetric. Hence, after replacing $G_0$ by a subgroup of index 2 if necessary, we conclude that $\chi|_{G_0 \times G_0}$ is alternating.
\end{proof}

\begin{proposition} \label{bicharacter cover by bilinear}
    Let $G$ be a locally compact abelian group. Then for every $\chi \in \Bic(G)$, there exists $(p,B) \in \Bil^\downarrow (G)$ such that $\chi \circ (p \times p)=\exp_\T \circ B$. In that case, we have $\rd B=\rd \chi$.
\end{proposition}
\begin{proof}
Take $(p,B) \in \Bil^\downarrow (G)$ such that $\rd B=\rd \chi$. Let $\chi'=\chi \circ (p \times p)$ and $\psi =\exp_\T \circ B$. Then 
$$\rd \chi'=\rd \chi \circ (\rd p \times \rd p)= \rd \psi.$$
Therefore, by Proposition \ref{bicharacter locally trivial}, we will have $\chi'=\psi$ after restricting $p$ to some open subgroup.
\end{proof}

\section{Cohomology of compact connected abelian groups} \label{integral cohomology}
In this section and all the remaining sections of this paper, all locally compact groups are assumed to be second countable.

We will need Moore's measurable group cohomology for locally compact groups \cite{Mo64, Mo75a, Mo75b}. A more recent reference is \cite{AM10}, which discusses the relation with other cohomology theories. Here, we recall the main definitions and properties.

Let $G$ be a locally compact group and $A$ a locally compact abelian group (Moore's cohomology is defined more generally when $A$ is a $G$-module, but for our purpose we don't this more general theory). 

For $n \in \N$, we let $\cC^n(G,A)$ denote the group of all Haar equivalence classes of measurable functions from $G^n$ to $A$. We define $\partial : \cC^n(G,A) \rightarrow \cC^{n+1}(G,A)$ by
$$ (\partial \varphi)(g_1,\dots,g_{n+1}) =  \varphi(g_2,\dots,g_{n+1}) + \sum_{i=1}^n (-1)^i f(g_1,\dots, g_ig_{i+1}, \dots, g_{n+1})+(-1)^{n+1} \varphi(g_1,\dots, g_{n}). $$
We have $\partial (\partial \varphi)=0$ for every $\varphi \in \cC^n(G,A)$. We define the group of $n$-cocycles $$Z^n(G,A)=\{ \varphi \in \cC^n(G,A) \mid \partial \varphi = 0\},$$ the group of $n$-coboundaries $$B^n(G,A)=\{ \varphi \in \cC^n(G,A) \mid \exists \psi \in \cC^{n-1}(G,A), \; \varphi=\partial \psi \}$$ and the $n$-cohomology group $ H^n(G,A)=Z^n(G,A)/B^n(G,A)$. Here we let $\cC^{0}(G,A)=A$ and $\cC^{-1}(G,A)=0$. Therefore $H^0(G,A)=A$. We also have $H^1(G,A)=Z^1(G,A)=\Hom(G,A)$.

Associated to the short exact sequence $\Z \rightarrow \R \rightarrow \T$, there is a long exact sequence 
$$ \dots \rightarrow H^n(G,\R) \rightarrow H^n(G,\T) \rightarrow H^{n+1}(G,\Z) \rightarrow H^{n+1}(G,\R) \rightarrow \dots $$
where the connecting morphism $\delta : H^n(G,\T) \rightarrow H^{n+1}(G,\Z)$, called the \emph{Bockstein morphism}, is given by
$\delta([c])=[\partial C]$ for every $c \in Z^n(G,\T)$ and $C \in C^n(G,\R)$ that satisfy $c=\exp_\T \circ \: C$.

\begin{theorem}[{\cite[Theorem A]{AM10}}]
    Let $K$ be a compact group. Then $H^n(K,\R)=0$ for all $n \geq 1$, hence $$\delta : H^n(K,\T) \rightarrow H^{n+1}(K,\Z)$$ is an isomorphism for all $n \geq 1$.

    In particular, if $K$ is abelian, we have an isomorphism 
    $$ \delta : \widehat{K}=H^1(K,\T) \rightarrow H^2(K,\Z).$$
\end{theorem}

Let $G$ be a locally compact group. Suppose that $A$ and $B$ are discrete abelian groups and let $A \otimes B$ denote their tensor product. We define the bilinear operation 
$$ \cC^n(G,A) \times  \cC^m(G,B) \ni (\varphi,\psi) \mapsto \varphi \wedge \psi \in \cC^{n+m}(G,A \otimes B)$$ by the formula
$$ (\varphi \wedge \psi)(g_1,\dots,g_{n+m})= \varphi(g_1,\dots,g_n) \otimes \psi(g_{n+1},\dots,g_{n+m}).$$
We have the following equation
$$ \partial(\varphi \wedge \psi)=(\partial \varphi) \wedge \psi + (-1)^n \varphi \wedge (\partial \psi)$$
which implies that $\wedge$ descends to a well-defined bilinear operation 
$$H^n(G,A) \times H^m(G,B) \ni (\alpha, \beta) \mapsto \alpha \wedge \beta \in H^{n+m}(G,A \otimes B).$$

Moreover, if $\tau : A \otimes B \rightarrow B \otimes A$ is the flip map, then $$\tau (\alpha \wedge \beta)=(-1)^{nm} \beta \wedge \alpha$$ for all $(\alpha,\beta) \in H^n(G,A) \times H^m(G,B)$.

If $A=B=\Z$, then $A \otimes B=\Z$ and $\wedge$ defines a \emph{graded ring} structure on $H^*(G,\Z)=\bigoplus_{n \in \N} H^n(G,\Z)$. Note that $H^*(G,\Z)$ is only \emph{graded-commutative} :
$$ \alpha \wedge \beta =(-1)^{nm} \beta \wedge \alpha$$
for all $(\alpha,\beta) \in H^n(G,\Z) \times H^m(G,\Z)$.

The integral cohomology ring $H^*(K,\Z)$ of a compact group $K$ can be computed with classical techniques because it coincides with the classifying space cohomology (see \cite{AM10}). In the specific case of compact connected abelian groups, we have the following theorem.
 
\begin{theorem}[{\cite[Theorem V.1.9]{HM73}}] \label{integral cohomology ring}
     Let $K$ be a second countable compact connected abelian group. Then $H^*(K,\Z)$ is freely generated as a graded ring by $H^2(K,\Z)$. In other words, $H^{2n+1}(K,\Z)=0$ and
     $$ H^2(K,\Z)^{\odot n} \ni \alpha_1 \odot \cdots \odot \alpha_n \mapsto \alpha_1 \wedge \cdots \wedge \alpha_n \in H^{2n}(K,\Z)$$
     is an isomorphism for all $n \geq 0$, where $H^2(K,\Z)^{\odot n}$ is the $n$th symmetric tensor power of $H^2(K,\Z)$.
\end{theorem}
\begin{corollary}
    Let $K$ be a compact connected abelian group. Then the map
    $$ \delta \wedge \delta : \widehat{K} \odot \widehat{K} \ni \gamma_1 \odot \gamma_2 \mapsto \delta (\gamma_1) \wedge \delta (\gamma_2) \in H^4(K,\Z)$$
    is an isomorphism.
\end{corollary}

We now give a different but closely related description of $H^4(K,\Z)$ that is more geometric and will be suitable for our applications to kernels on von Neumann algebras. 

Let $ p : G' \rightarrow G$ be a covering map between two locally compact abelian groups. We can associate to $p$ a canonical cohomology class $c_p \in H^2(G,\ker p)$ defined by $c_p=[\partial s]$ where $s : G \rightarrow G'$ is any measurable section of $p$ viewed as an element of $\cC^1(G,G')$, so that $\partial s \in Z^2(G,\ker p)$. The cohomology class $c_p$ does not depend on the choice of $s$.

\begin{theorem} \label{beta isomorphism}
    Let $K$ be a compact connected abelian group. 
    \begin{enumerate}
    \item Take $\cb \in \Bint(K)$ and $(p,B) \in \Bil_{\rm int}^\downarrow(K)$ such that $\rd B=\cb$. Let $c_p \in H^2(K,\ker p)$ be the cohomology class associated to $p$. Then the cohomology class $$\beta_0(\cb):=B(c_p \wedge c_p) \in H^4(K,\Z)$$ depends only on $\cb$ and not on the choice of $(p,B)$.
    \item The map 
    $$ \beta_0 : \Bint(K) \rightarrow H^4(K,\Z)$$
    is a group morphism.
    \item $\beta_0$ vanishes on anti-symmetric bilinear maps, hence it
    induces a quotient morphism
    $$ \beta : \Qint(K) \rightarrow H^4(K,\Z).$$ 
    \item The morphism $\beta$ satisfies $\beta \circ (\rd \odot \rd) = \delta \wedge \delta$ where 
    $$\rd \odot \rd : \widehat{K} \odot \widehat{K} \rightarrow \Qint(K)$$ is the isomorphism of Proposition \ref{isomorphism square quadratic}. In particular, $\beta$ is an isomorphism.
    \end{enumerate}
\end{theorem}
\begin{proof}
(1) First, observe that $B(c_p \wedge c_p) \in H^4(K,\Z)$ is well defined because $B$ is integral so it defines a morphism from $\ker p \otimes \ker p$ to $\Z$.

Suppose that $(p',B') \in \Bil_{\rm int}^{\downarrow}(K)$ covers $(p,B)$. Take $u \in \Hom(p',p)$ such that $B'=B \circ (u \times u)$. Let $s'$ be a measurable section of $p'$ and let $s=u \circ s'$. Then $s$ is a measurable section of $p$ hence $c_{p}=u \circ c_{p'}$. Therefore $B(c_p \wedge c_p) = B'(c_{p'} \wedge c_{p'})$. We conclude that $B(c_p \wedge c_p)$ does not depend on the choice of $(p,B)$ up to local equivalence. By Proposition \ref{local equivalence and tangent form}, this means that $B(c_p \wedge c_p)$ depends only on $\rd B=\cb \in \Bint(K)$. 

    (2) Take $\cb, \mathfrak{d} \in \Bint(K)$. Take $(p,B)$ and $(q,D)$ in $\Bil_{\rm int}^\downarrow(K)$ such that $\rd B=\cb$ and $\rd D=\mathfrak{d}$. Up to replacing $p$ and $q$ by their fiber products $p \times_K q$, we may assume that $p=q$. Then we have
    $$\beta_0(\cb)+\beta_0(\mathfrak{d})=B(c_p \wedge c_p)+D(c_p \wedge c_p)=(B+D)(c_p \wedge c_p)=\beta_0(\cb+\mathfrak{d}).$$
    This shows that $\beta_0$ is a group morphism.
    
    (3) Let $\tau$ be the flip map on $\ker p \otimes \ker p$. Then we have $\tau( c_p \wedge c_p) =(-1)^4 c_p \wedge c_p=c_p \wedge c_p$. Therefore, we have
     $$\beta_0(\cb)=B(c_p \wedge c_p) =(B\circ \tau)(c_p \wedge c_p)  = B^{\op}(c_p \wedge c_p)=\beta_0(\cb^{\op}).$$
     We conclude by Proposition \ref{anti-symmetric is difference} that $\beta_0$ vanishes on anti-symmetric bilinear maps.

  (4)  Take $\gamma_1,\gamma_2 \in \widehat{K}$. There exists a covering $p : G \rightarrow K$ and two linear forms $\phi_i$ on $G$, such that $\gamma_i \circ p=\exp_\T \circ \ \phi_i$. Then $B=\phi_1 \otimes \phi_2 : (g,h) \mapsto \phi_1(g)\phi_2(h)$ is an integral covering bilinear form on $G$ and we have $\rd B=\rd \gamma_1 \otimes \rd \gamma_2$. Moreover, we have $B(c_p \wedge c_p)=\phi_1(c_p)\wedge \phi_2(c_p)$. 
    
    Let $s : K \rightarrow G$ be a measurable section of $p$. Since $\exp_\T \circ \ \phi_i \circ s=\gamma_i$, we have $$\delta(\gamma_i) =[\partial (\phi_i \circ s)]=[\phi_i \circ \partial s]= \phi_i(c_p).$$ We conclude that
    $$ \beta(\rd \gamma_1 \odot \rd \gamma_2)=\delta(\gamma_1) \wedge \delta(\gamma_2).$$
\end{proof}

\begin{remark}
    Let $K$ be a torus. Let $c \in H^2(K,\ker \exp_K)$ be the cohomology class associated to the universal covering $\exp_K : \Lie(K) \rightarrow K$.
    Then the group morphism $$\beta_0 : \Bint(K) \rightarrow H^4(K,\Z)$$ is given by the simple formula $\beta_0(\cb)= \cb(c \wedge c)$ for all $\cb \in \Bint(K)$ (see Remark \ref{remark covering form on torus}).
\end{remark}

\section{Compact connected abelian kernels on von Neumann algebras}
All von Neumann algebras in this paper are assumed to have separable predual. For a von Neumann algebra $M$, we denote by $\epsilon_M$ the quotient morphism from $\Aut(M)$ to $\Out(M)=\Aut(M)/\Inn(M)$ where $\Inn(M)=\{ \Ad(u) \mid u \in \cU(M)\}$ is the group of inner automorphisms. Following \cite{Co74}, we say that $M$ is full if $\Inn(M)$ is closed in $\Aut(M)$. In that case, $\Out(M)$ becomes a Polish group and $\epsilon_M$ is continuous.

Let $G$ be a locally compact group. Following \cite[Section 3]{Su80}, a \emph{$G$-kernel} on $M$ is a morphism $\kappa : G \rightarrow \Out(M)$ that admits a measurable lift to $\Aut(M)$. When $M$ is full, this is equivalent to asking that $\kappa$ is continuous, because a measurable morphism between two polish groups is automatically continuous. Observe that if $M$ is not full, then $\ker \kappa$ is not necessarily closed in $G$ (see Theorem \ref{trivial isotropic vector} for instance).
We say that $\kappa$ is \emph{split} if there exists a continuous morphism $\rho : G \rightarrow \Aut(M)$ such that $\kappa=\epsilon_M \circ \rho$. In that case, we say that $\rho$ is a \emph{$\kappa$-splitting action}.

The following lemma is extracted from the proof of \cite[Theorem 4.1.3]{Su80}.
\begin{lemma} \label{infinite vanishing cocycle}
 Let $G$ be a locally compact abelian group and $\rho : G \rightarrow \Aut(M)$ an action on some von Neumann algebra $M$. Suppose that $M$ is properly infinite. Take $c \in Z^2(G, \rho, \cU(M))$ be a $2$-cocycle, i.e.\  $c : G \times G \rightarrow \cU(M)$ is a measurable map such that
 $$ \forall (g,h,k) \in G^3, \quad  c(g,h) c(gh,k)= \rho(g)(c(h,k)) c(g,hk) $$
 Then $c \in B^2(G,\rho,\cU(M))$, i.e.\ there exists a measurable map $v : G \rightarrow \cU(M)$ such that
 $$ \forall (g,h) \in G^2, \quad v_g \rho(g)(v_h)=c(g,h) v_{gh}.$$
\end{lemma}

We will need the following lemma due to Moore.
\begin{lemma} \label{symmetric vanishing cocycle}
Let $G$ be a locally compact abelian group and $A$ an abelian von Neumann algebra. Then any symmetric $2$-cocycle $c \in Z^2(G, \cU(A))$ is a $2$-coboundary. The same holds if we replace $\cU(A)$ by $\cU(A)/\T$.
\end{lemma}
\begin{proof}
Let $c : G \times G \rightarrow \cU(A)$ be a symmetric $2$-cocycle, i.e.\ $c(g,h)=c(h,g)$ for all $g,h \in G$. Following the proof of \cite[Lemma 2.6(2)]{ISW20}, denote by $H = \mathcal U(A) \rtimes_c G$ the group extension associated with the $2$-cocycle $c$.  Then $H$ is a Polishable group such that the short exact sequence $1 \to \mathcal U(A) \to H \to G \to 1$ consists of continuous maps (see \cite[Theorem 10]{Mo75a}). Since $c$ is symmetric, $H$ is abelian and \cite[Theorem 4]{Mo75b} implies that $c$ is a $2$-coboundary.

Let $c : G \times G \rightarrow \cU(A)/\T$ be a symmetric $2$-cocycle. Then the exact same proof as in \cite[Lemma 6.3(3)]{ISW20} with $\R$ replaced by $G$ shows that $c$ is a $2$-coboundary.
\end{proof}

Let $G$ be a locally compact abelian group. It is easy to check that for every $c \in Z^2(G,\T)$, its antisymmetrization $c^{(2)} : (g,h) \mapsto c(g,h) \overline{c(h,g)}$ is a bicharacter on $G$. We define the group $$D(G) \dot{=} \{ \chi \in \Bic(G) \mid \exists c \in Z^2(G,\T), \: \chi=c^{(2)} \}.$$
By the previous lemma appplied to $A=\C$, the map
$$ H^2(G,\T) \ni [c] \mapsto c^{(2)} \in D(G)$$
is a group isomorphism. The group
$D(G)$ is contained in the group of all alternating bicharacters $$A(G)=\{ \chi \in \Bic(G) \mid \forall g \in G, \: \chi(g,g)=1\}.$$ 
On the other hand, since $\Bic(G) \subset Z^2(G,\T)$, then $D(G)$ contains the group 
$$B^{(2)}(G)=\{ \chi^{(2)} \mid \chi \in \Bic(G) \}.$$ In some cases, for example when $G$ is connected, we have 
$A(G)=B^{(2)}(G)$, hence $D(G)=A(G)=B^{(2)}(G)$. We refer the reader to \cite{Kl65} for more details on this. Since coverings of compact connected groups are not necessarily connected, we care about the case where $G$ is not connected, for which we do not have $D(G) = A(G)$ in general.

\begin{lemma} \label{bicharacter from kernel}
    Let $G$ be a locally compact abelian group and let $\lambda : G \rightarrow \Out(M)$ be a split kernel on some factor $M$. Suppose that $N$ is a closed subgroup of $G$ such that $N \subset \ker \lambda$ and let $\kappa : G/N \rightarrow \Out(M)$ be the quotient kernel.
   \begin{enumerate}
        \item For every $\lambda$-splitting action $\rho : G \rightarrow \Aut(M)$, there exists a unique bicharacter $\chi_\rho \in \Bic(G,N)$ such that for all $g \in G$, all $n \in N$ and all $u \in \cU(M)$ with $\Ad(u)=\rho(n)$, we have
        $$\rho(g)(u)=\chi_\rho(g,n)u.$$
        Moreover, the bicharacter $\chi_\rho|_{N \times N} \in \Bic(N,N)$ is alternating.
       \item If $\rho,\pi : G \rightarrow \Aut(M)$ are two $\lambda$-splitting actions, there exists $\zeta \in D(G)$ such that $$\chi_\pi \chi_\rho^{-1}=\zeta |_{G \times N}.$$
       \item If $M$ is infinite, $\rho$ is a $\lambda$-splitting action and $\zeta \in D(G)$, there exists a $\lambda$-splitting action $\pi$ such that $$\chi_\pi \chi_\rho^{-1}=\zeta|_{G \times N}.$$
       \item If $\kappa$ is split then there exists $\zeta \in D(G)$ such that $$\chi_\rho = \zeta|_{G \times N}$$ and the converse is also true if $M$ is infinite.
   \end{enumerate}
\end{lemma}
\begin{proof}
    (1) If $n \in N$ and $\rho(n)=\Ad(u)$ then $\rho(n)(u)=u$, hence $\chi_\rho(n,n)=1$.
    
    (2) Write $\pi(g)=\Ad(v_g) \circ \rho(g)$ for all $g \in G$ with $g \mapsto v_g \in \cU(M)$ measurable. Observe that $v_g\rho(g)(v_h)=c(g,h)v_{gh}$ for some $2$-cocycle $c \in Z^2(G,\T)$. Let $\zeta=c^{(2)} \in D(G)$. Then, we have
    $$v_g\rho(g)(v_h) = \zeta(g,h) v_h\rho(h)(v_g)$$
    for all $g,h \in G$.

    Now, write $\rho(n)=\Ad(u_n)$ for all $n \in n$. 
    We have $\pi(n)=\Ad(w_n)$ where $w_n=v_n u_n$. We then compute
    \begin{align*}
        \pi(g)(w_n)& =v_g\rho(g)(v_n u_n)v_g^* \\
        &=\chi_\rho(g,n) v_g\rho(g)(v_n)u_n v_g^* \\
        &= \chi_\rho(g,n) \zeta(g,n) v_n \rho(n)(v_g)u_n v_g^* \\
        &= \chi_\rho(g,n) \zeta(g,n) v_n (u_n v_g u_n^*) u_n v_g^* \\
        & =  \chi_\rho(g,n) \zeta(g,n)  w_n.
    \end{align*}
    We conclude from this computation that  $\chi_\pi=\chi_\rho  \cdot \zeta|_{G \times N}$.

    (3) Take $ \zeta \in D(G)$ and write $ \zeta =c^{(2)}$ for some $2$-cocycle $c \in Z^2(G,\T)$. View $c$ as an element of $Z^2(G,\rho,\cU(M))$. Then, since $M$ is infinite, Lemma \ref{infinite vanishing cocycle} tells us that $c \in B^2(G,\rho,\cU(M))$, i.e.\ there exists a measurable map $v : g \mapsto v_g \in \cU(M)$ such that $v_g \rho(g)(v_h)=c(g,h) v_{gh}$. Let $\pi(g)= \Ad(v_g) \circ \rho(g)$. This defines a $\lambda$-splitting action $\pi : G \rightarrow \Aut(M)$ and by the previous calculations, we will have $\chi_\pi =\chi_\rho \cdot \zeta |_{G \times N}$.

   (4) The first part follows directly from item (2). Conversely, suppose that $M$ is infinite and that $\chi_\rho = \zeta |_{G \times N}$ for some $\zeta \in D(G)$. We want to show that $\kappa$ is split. By item (3), we may assume that $\chi_\rho$ is actually trivial. For every $n \in N$, write $\rho(n)=\Ad(u_n)$ for some $u_n \in \cU(M)$. Since $\chi_\rho$ is trivial, the unitaries $(u_n)_{n \in N}$ are all fixed by $\rho$ and they all commute. Thus, they generated an abelian von Neumann algebra $A \subset M^\rho$. Take $s : G/N \rightarrow G$ a borel section of the quotient map $p : G \rightarrow G/N$. Let $\theta = \rho \circ s$. Then we have $\theta(kh)=\Ad(u(k,h)) \theta(k)\theta(h)$ for some unitaries $u(k,h) \in A$, for all $k,h \in G/N$. We can view $(k,h) \mapsto \Ad(u(k,h))$ as a symmetric 2-cocycle with values into $\cU(A)/\T$. By Lemma \ref{symmetric vanishing cocycle}, this cocycle is necessarily a coboundary, i.e.\ there exists a measurable map $G/N \ni k \mapsto v_k \in \cU(A)$ such that $k \mapsto \Ad(v_k) \circ \theta(k)$ is a true group morphism from $G/N$ to $\Aut(M)$. We conclude that $\kappa$ is split.

\end{proof}

If $G=\R^n$ and $N < G$ is a closed cocompact subgroup, then every bicharacter in $\Bic(G,N)$ extends uniquely to $G \times G$ and is of the form $\exp_\T \circ B|_{G \times N}$ for a unique $B \in \Bil(G)$. In general, for coverings of arbitrary compact abelian groups, we will need the following proposition.
\begin{proposition} \label{cocompact bicharacter}
    Let $G$ be a locally compact abelian group and $N < G$ a closed cocompact subgroup. Take $\chi \in \Bic(G,N)$.
    \begin{enumerate}
    \item There exists a covering $p : G' \rightarrow G$ and $B \in \Bil(G')$ such that $\chi \circ (p \times p) =\exp_\T \circ B|_{G' \times N'}$ where $N'=p^{-1}(N)$.
    \item The bilinear form $\rd' \chi :=\rd B \in \Bil(\Lie(G))$ depends only on $\chi$ and not on the choice of $(p,B)$ in (1).
    \item If $\chi$ extends to a bicharacter $\zeta \in \Bic(G)$ then $\rd' \chi=\rd \zeta$.
    \item The map $\chi \mapsto \rd' \chi$ is a group morphism from $\Bic(G,N)$ to $\Bil(\Lie(G))$.
    \end{enumerate}
\end{proposition}
\begin{proof}
    (1) Let $\iota : N \rightarrow G$ be the inclusion of $N$ in $G$. Since $N$ is closed and cocompact in $G$, the dual map $\widehat{\iota} : \widehat{G} \rightarrow \widehat{N}$ is a covering map. Write $\chi(g,n)=\langle \theta(g),n \rangle$ with $\theta \in \Hom(G,\widehat{N})$.  Let $\widehat{\iota} \times_{\widehat{N}} \theta : G' \rightarrow \widehat{N}$ be the fiber product of $\widehat{\iota}$ and $\theta$. Let $\theta'$ and $p \in \Hom( \widehat{\iota} \times_{\widehat{N}} \theta, \theta)$ be the first and second coordinate projection. Then $p : G' \rightarrow G$ is a covering and $\theta' \in \Hom(G',\widehat{G})$ satisfies $\widehat{\iota} \circ \theta' = \theta \circ p$. Then the bicharacter $\chi' \in \Bic(G')$ defined by $\chi'(g,h)=\langle \theta'(g), p(h)\rangle$ satisfies $\chi'|_{G' \times N'}=\chi \circ (p \times p)$, where $N'=p^{-1}(N)$. Now, thanks to Proposition \ref{bicharacter cover by bilinear}, up to replacing $p$ by $p \circ q$ and $\chi'$ by $\chi' \circ (q \times q)$ for some covering $q$ of $G'$, we may assume that $\chi' = \exp_\T \circ B$ for some $B$ as we wanted.
    
    (2) Let $(p_i,B_i) \in \Bil^\downarrow(G), i=1,2$ satsifying (1). Up to replacing $p_1$ and $p_2$ by $p=p_1 \times_G p_2$ and replacing $B_i$ by $B_i \circ (r_i \times r_i)$ where $r_i \in \hom(p,p_i)$ are the coordinate projections, we may assume that $p_1=p_2=p : G' \rightarrow G$. Let $D=B_1-B_2$. Then $\exp_\T \circ D$ is trivial on $G' \times N'$. Thus $D$ takes only integral values on $G' \times N'$. Since $D$ is continuous and vanishes on $\{1\} \times N'$, this implies that $D$ vanishes on $G'_c \times N'$ where $G'_c$ is the connected component of $G'$. Since $N'$ is cocompact in $G'$, this in turns implies that $D$ vanishes on $G'_c \times G'$. In particular, $\rd D=0$ and we conclude that $\rd B_1=\rd B_2$.
    
    (3) Thanks to Proposition \ref{bicharacter cover by bilinear}, we can find $(p,B) \in \Bil^\downarrow(G)$ such that $\zeta \circ (p \times p)=\exp_\T \circ B$ and $\rd B=\rd \zeta$. By (2), we conclude that $\rd' \chi=\rd B=\rd \zeta$.

    (4) For $i=1,2$, let $\chi_i \in \Bic(G,N)$ and $(p_i,B_i) \in \Bil^\downarrow(G)$ satisfying (1). Let $p=p_1 \times_G p_2$ and $B=B_1 \circ (r_1 \times r_1) + B_2 \circ (r_2 \times r_2)$ where $r_i \in \hom(p,p_i)$ are the coordinate projections. Then $(p,B)$ satisfies (1) for $\chi=\chi_1+\chi_2$. Moreover, $\rd B = \rd B_1+ \rd B_2$. We conclude that $\rd' \chi = \rd' \chi_1+\rd' \chi_2$.
\end{proof}

\begin{definition}
    Let $K$ be a compact connected abelian group and let $\kappa : K \rightarrow \Out(M)$ be a kernel on some factor $M$. 
    \begin{itemize}
        \item We say that $\kappa$ is \emph{locally split} if there exists a covering $p : G \rightarrow K$ such that $\kappa \circ p$ is split.
        \item  If $p : G \rightarrow K$ a covering of $K$ and $\rho : G \rightarrow \Aut(M)$ is a splitting action for $\kappa \circ p$, we say that $(p,\rho)$ is a \emph{$\kappa$-covering action}. 
        \item Let $(p,\rho)$ be a $\kappa$-covering action. We define $\chi_\rho \in \Bic(G,\ker p)$ as in Proposition \ref{bicharacter from kernel} and we let $\cb_\rho := \rd'\chi_\rho$ as in Proposition \ref{cocompact bicharacter}. By identifying $\Lie(G)$ with $\Lie(K)$, we view $\cb_\rho$ as an element of $\Bil(\Lie(K))$. Our notation is slightly abusive because $\cb_\rho$ depends on $(p,\rho)$ and not only on $\rho$.
    \end{itemize}
\end{definition}

\begin{remark}
    Let $\kappa : K \rightarrow \Out(M)$ be a kernel and $(p,\rho)$ a $\kappa$-covering action. If $\kappa$ is injective then $p$ is uniquely determined by $\rho$. However, we will be interested in non-injective kernels.
\end{remark}

\begin{theorem} \label{quadratic from kernel}
    Let $K$ be a compact connected abelian group. Let $\kappa : K \rightarrow \Out(M)$ be a kernel on a factor $M$. Suppose that $\kappa$ is locally split.
    \begin{enumerate}
        \item The quadratic form $\cq_\kappa \in \Quad(\Lie(K))$ defined by $$\cq_\kappa(\xi)=\cb_\rho(\xi,\xi), \quad \xi \in \Lie(K)$$
        does not depend on the choice of the $\kappa$-covering action $(p,\rho)$.
        \item We have $\cq_\kappa \in \Qint(K)$.
        \item If $M$ is infinite and $\cb \in \Bil(\Lie(K))$ satisfies $\cq_\kappa(\xi)=\cb(\xi,\xi)$
        for all $\xi \in \Lie(K)$, then there exists a $\kappa$-covering action $(p,\rho)$ such that $\cb_\rho=\cb$.
        \item If $\kappa$ is split then $\cq_\kappa=0$, and the converse is also true if $M$ is infinite.
    \end{enumerate}
\end{theorem}
\begin{proof}
    (1)  Let $(p_i,\rho_i)$ $i=1,2$ be two covering actions for $\kappa$. Let $p=p_1 \times_K p_2$ be the fiber product of $p_1$ and $p_2$ and let $r_i \in \Hom(p, p_i)$ be the coordinate projections. Let $\sigma_i=\rho_i \circ r_i$. Observe that $\chi_{\sigma_i} = \chi_{\rho_i} \circ (r_i \times r_i)$ hence $\cb_{\sigma_i} = \cb_{\rho_i}$.
    Moreover, we have
    $$\epsilon_M \circ \sigma_i= \epsilon_M \circ \rho_i \circ r_i=\kappa \circ p_i \circ r_i= \kappa \circ p.$$
    Therefore, by Lemma \ref{bicharacter from kernel}, $\chi_{\sigma_1} \overline{\chi_{\sigma_2}}$ extends to some bicharacter $\zeta \in D(G)$. Proposition \ref{cocompact bicharacter} then shows that $$ \cb_{\sigma_1} - \cb_{\sigma_2}=\rd' \chi_{\sigma_1} - \rd' \chi_{\sigma_2}=\rd \zeta.$$ Since $\zeta$ is alternating, $\rd \zeta$ is anti-symmetric and we conclude that $\cb_{\rho_1}$ and $\cb_{\rho_2}$ define the same quadratic form on $\Lie(K)$.

    (2) Let $(p,\rho)$ be a $\kappa$-covering action. Thanks to Proposition \ref{cocompact bicharacter}, we may assume that $\chi_\rho = \exp_\T \circ B |_{G \times \ker p}$ for some $B \in \Bil(G)$ where $G$ is the covering group of $p$. By Proposition \ref{bicharacter from kernel}, we know that $\chi_\rho|_{\ker p \times \ker p}$ is alternating, i.e.\ $\chi_\rho(n,n)=1$ for all $n \in \ker p$. This means that $B(n,n) \in \Z$ for all $n \in \ker p$. In other words $(p,Q) \in \Quad_{\rm int}^\downarrow(K)$ where $Q(g)=B(g,g)$ for all $g \in G$. We conclude that $\cq_\kappa=\rd Q \in \Qint(K)$ by Proposition \ref{integral covering quadratic form}.
    
    (3) Take $p : G \rightarrow K$ a covering and $\rho : G \rightarrow \Aut(M)$ a $(\kappa \circ p)$-splitting action. Then $\cb - \cb_\rho$ is anti-symmetric. Thus, we can write $\cb -\cb_\rho = \cd - \cd^{\op}$ for some $\cd \in \Bil(\Lie(K))$. We can find a covering bilinear form $(q,D) \in \Bil^\downarrow(K)$ such that $\rd D = \cd$. Up to replacing $p$ and $q$ by their fiber product, we may assume that $p=q$. Then $\exp_\T \circ D \in \Bic(G) \subset Z^2(G,\T)$. Thus $\zeta = \exp_\T \circ (D-D^{\op}) \in D(G)$. We then know that there exists a $(\kappa \circ p)$-splitting action $\pi : G \rightarrow \Aut(M)$ such that $\chi_\pi \overline{\chi_\rho} = \zeta|_{G \times \ker p}$. Then we have 
    $$\cb_\pi -\cb_\rho = \rd' \chi_\pi-\rd' \chi_\rho= \rd \zeta=\rd D- \rd D^{\op}=\cd - \cd^{\op}=\cb- \cb_\rho.$$
    We conclude that $\cb_\pi=\cb$.

    (4) If $\kappa$ splits then it can be lifted to an action $\rho : K \rightarrow \Aut(M)$ for which $\chi_\rho$ is trivial by definition, hence $\cq_\kappa=0$. 
    
    Conversely, suppose that $\cq_\kappa=0$. Let $p: G \rightarrow K$ be a covering and $\rho : G \rightarrow \Aut(M)$ be a $\kappa \circ p$-splitting action. Thanks to Proposition \ref{cocompact bicharacter}, we may assume that $\chi_\rho = \exp_\T \circ B |_{G \times \ker p}$ for some $B \in \Bil(G)$. Since $\cq_{\kappa}=0$, then $\cb_\rho$ is anti-symmetric. Since $\cb_\rho = \rd B$, Proposition \ref{locally anti-symmetric} shows that, up to replacing $G$ by some open subgroup, we may assume that $B$ is anti-symmetric hence $B=D-D^{\op}$ for some $D \in \Bil(G)$. Then, like in the previous item, we obtain $\exp_\T \circ B \in D(G)$. We conclude by Proposition \ref{bicharacter from kernel} that $\kappa$ splits if $M$ is infinite.
\end{proof}

%\begin{theorem}[{\cite{Va05,MV23}}] Let $G$ be a locally compact group. Then there exists a $\II_1$ factor $M$ an an action $\rho : G \rightarrow \Aut(M)$ such that $\rho$ is strictly outer, in the sense that $$M' \cap (M \rtimes_\rho G)=\C.$$ Moreover, we can choose $M$ to be a hyperfinite factor or a free group factor. \end{theorem}

\begin{theorem} \label{kernel from quadratic}
     Let $K$ be a compact connected abelian group and let $\cq \in \Qint(K)$. Then there exists a locally split kernel $\kappa : K \rightarrow \Out(M)$ on some $\II_1$ factor $M$ such that $\cq_\kappa=\cq$.

     Moreover, we can take $M$ to be either hyperfinite or full.
\end{theorem}
\begin{proof}
    Let $\cb \in \Bil(\Lie(L))$ be an integral bilinear form such that $\cb(\xi,\xi)=\cq(\xi)$ for all $\xi \in \Lie(K)$. By Proposition \ref{integral covering of integral bilinear form}, there exists a covering group $p : G \rightarrow K$ and an integral bilinear form $B \in \Bil(G)$ such that $\rd B=\cb$. Let $\chi=\exp_\T \circ B$. Note that since $K$ is connected, then $p$ is surjective, hence $p$ descends to a group isomorphism from $G/\Lambda$ onto $K$ where $\Lambda=\ker p$.
    
    Suppose that we have a $\II_1$ factor with an outer action $\theta : G \curvearrowright Q$ and let $\theta_0=\theta|_{\Lambda}$. Consider the twisted crossed product $M=Q \rtimes_{\theta_0} \Lambda$. Observe that $M$ is a factor because $\theta_0$ is outer.

    Let $\rho : G \rightarrow \Aut(M)$ be the unique extension of $\theta$ from $Q$ to $M$ such that $\rho(g)(u_\lambda)=\chi(g,\lambda)u_\lambda$ for all $\lambda \in \Lambda$ and $g \in G$. Since $B$ is integral, $\chi|_{\Lambda \times \Lambda}$ is trivial, hence $\rho(\lambda)=\Ad(u_\lambda) \in \Inn(M)$ for all $\lambda \in \Lambda$. Therefore $\rho$ is a covering action of some kernel $\kappa : G/\Lambda=K \rightarrow \Aut(M)$ and by construction, we have $\cq_\kappa(\xi)=\rd \chi(\xi,\xi)=\cb(\xi,\xi)$ for all $\xi \in \Lie(K)$.
    
    %Observe that $\kappa$ is injective. Indeed, if $\rho(g)=\Ad(u)$ for some $u \in M$, then $u$ must be in the normalizer of $Q$ which is equal to $\cU(Q)\cdot \{ u_\lambda \mid \lambda \in \Lambda \}$ and this means that $\theta(g)=\rho(g)|_Q \in \Inn(Q) \cdot \{ \theta(\lambda) \mid \lambda \in \Lambda \}$, but this is only possible if $g \in \Lambda$ because $\theta$ is outer.

    Finally, we can take $Q$ to be the hyperfinite $\II_1$ factor and $\theta : G \curvearrowright Q$ a Bogoljubov action associated to some faithful representation $\pi$ (see \cite{Pl79}). Then $M$ will be a hyperfinite $\II_1$ factor. 
    
   We can also take $Q$ to be a free group factor and $\theta : G \curvearrowright Q$ a free Bogoljubov action associated to some faithful mixing representation $\pi$. Since $\pi$ is mixing, the morphism $\epsilon_Q \circ \theta_0 : \Lambda \rightarrow \Out(Q)$ will have a discrete image. Thus $M$ will be a full factor thanks to \cite{Jo81} (see also \cite{Hou12}).
\end{proof}

\begin{proposition} \label{operation on kernels}
    Let $K$ and $L$ be compact connected abelian groups. 
    \begin{enumerate}
        \item If $\kappa : K \rightarrow \Out(M)$ is a locally split kernel on a factor $M$ and $\theta \in \Hom(L,K)$, then $\cq_{\kappa \circ \theta}=\cq_\kappa \circ \rd \theta$.
        \item If $\kappa_i : K \rightarrow \Out(M_i)$ are two locally split kernels on some factors $M_i$ for $i=1,2$, then the kernel 
        $$\kappa_1 \otimes \kappa_2 : K \rightarrow \Out(M_1 \ovt M_2)$$ is also locally split and we have $\cq_{\kappa_1 \otimes \kappa_2}=\cq_{\kappa_1} +\cq_{\kappa_2}$.
        \item If $\kappa : K \rightarrow \Out(M)$ is a 
        locally split kernel on a factor $M$ and $\kappa^{\rm op} : K \rightarrow \Out(M^{\rm op})$ is its opposite kernel then $\cq_{\kappa^{\rm op}}=-\cq_\kappa$.
    \end{enumerate}
\end{proposition}
\begin{proof}
(1) Let $p : G \rightarrow K$ be a covering and $\rho : G \rightarrow K$ be a $(\kappa \circ p)$-splitting action. Let $q \in \Hom_K(\theta \times_K p,\theta)$ be the first coordinate projection and $\theta' \in \Hom_K(\theta \times_K p,p)$ the second coordinate projection. Then $q$ is a covering of $L$ and $p \circ \theta' = \theta \circ q$. Let $\eta=\rho \circ \theta'$. Then $(q,\eta)$ is a covering action for $\kappa \circ \theta$ and it is easy to check that $\chi_{\eta}=\chi_\rho \circ (\theta \times \theta )$, hence $\cb_{\eta}=\cb_{\rho} \circ \rd \theta$. We conclude $\cq_{\kappa \circ \theta}=\cq_{\kappa} \circ \rd \theta$.

(2) Let $(p_i, \rho_i)$ be a covering action for $\kappa_i$. Let $p=p_1 \times_K p_2$ with the coordinate projections $r_i \in \Hom(p,p_i)$. Set $\eta_i=\rho_i \circ r_i$. Then $(p, \eta_1 \otimes \eta_2)$ is a covering action for $\kappa_1 \otimes \kappa_2$ and it is easy to check that $\chi_{\eta_1 \otimes \eta_2}=\chi_{\eta_1} \cdot \chi_{\eta_2}$, hence
$$\cb_{\eta_1 \otimes \eta_2}=\cb_{\eta_1} +\cb_{\eta_2}=\cb_{\rho_1} + \cb_{\rho_2}$$
from which we conclude that $\cq_{\kappa_1 \otimes \kappa_2}=\cq_{\kappa_1} +\cq_{\kappa_2}$.

(3) We have $\kappa^{\op}(g)=\kappa(g)^{\op}$ for all $g \in K$. Let $p : G \rightarrow K$ be a covering group and $\rho : G \rightarrow \Aut(M)$ a $(\kappa \circ p)$-splitting action. Then $\rho^{\op} : G \rightarrow \Aut(M^{\op})$ is a $(\kappa^{\op} \circ p)$-splitting action. Take $g \in G$, $n \in \ker p$ and $u \in \cU(M)$ such that $\rho(n)=\Ad(u)$. Then $\rho^{\op}(n)=\Ad(u)^{\op}=\Ad((u^*)^{\op})$. Thus, $$\chi_{\rho^{\op}}(g,n) =\rho(g)^{\op}((u^*)^{\op})u^{\op} =  (\rho(g)(u^*))^{\op} u^{\op}=(u \rho(g)(u^*))^{\op}= \overline{\chi_\rho(g,n)}.$$
    This shows that $\chi_{\rho^{\op}}=\overline{\chi_\rho}$, hence $\cb_{\rho^{\op}}=-\cb_\rho$ and $\cq_{\kappa^{\op}}=–\cq_\kappa$.
\end{proof}

 \section{Computing the 3-cohomology obstruction} \label{section 3-cocycle}
 Let $G$ be a locally compact group and $\kappa : G \rightarrow \Out(M)$ a kernel on a factor $M$. The \emph{obstruction} associated to $\kappa$ is a cohomology class $\Ob(\kappa) \in H^3(G,\T)$ defined in \cite[Section 3]{Su80}. It is represented by the $3$-cocycle
 $$  c : (g,h,k)  \mapsto \alpha(g)(u(h,k)) u(g,hk)  u(gh,k)^{-1}u(g,h)^{-1}$$
 where $\alpha : G \rightarrow \Aut(M)$ is a measurable lift of $\kappa$ and $u : (g,h) \mapsto u(g,h) \in \cU(M)$ is a measurable map such that
 $$ \forall (g,h) \in G^2, \quad \alpha_g \circ \alpha_h = \Ad(u(g,h)) \circ \alpha_{gh}.$$
Sutherland shows that if $M$ is infinite and $\Ob(\kappa)=0$, then $\kappa$ splits \cite[Theorem 4.1.3]{Su80}.

Our goal in this section is to compute $\Ob(\kappa)$ and relate it to the quadratic form $\cq_\kappa$ that we defined int the previous section when $\kappa$ is a locally split kernel of a compact connected abelian group. Before digging into the computations, we first observe, by an abstract argument, that we necessarily have a strong relation between these two invariants.

  \begin{proposition} \label{alpha quadratic to obstruction}
     Let $K$ be a compact connected abelian group. There exists a group morphism 
     $$\alpha : \Qint(K) \rightarrow H^3(K,\T)$$
     such that for every locally split kernel $\kappa : K \rightarrow \Out(M)$ on a factor $M$, we have $\Ob(\kappa)=\alpha( \cq_\kappa)$.

     Moreover, $\alpha$ is injective.
   \end{proposition}
   \begin{proof}
      Let $\kappa_i : K \rightarrow \Out(M_i)$ be two locally split kernels on infinite factors $M_i$ such that $\cq_{\kappa_1}=\cq_{\kappa_2}$. By Proposition \ref{operation on kernels}, we have $$\cq_{\kappa_1 \otimes \kappa_2^{\rm op}}=\cq_{\kappa_1} +  \cq_{\kappa_2^{\rm op}}=\cq_{\kappa_1} - \cq_{\kappa_2}  = 0.$$ Thus $\kappa_1 \otimes \kappa_2^{\rm op}$ is split. Therefore, by \cite[Proposition 3.2.1 and 3.2.3]{Su80}, we have
       $$ \Ob(\kappa_1) \Ob(\kappa_2)^{-1}=\Ob(\kappa_1 \otimes \kappa_2^{\rm op})=1.$$
       We conclude that $\Ob(\kappa_1)=\Ob(\kappa_2)$. 
    
    This means that for every $\cq \in \Qint(K)$, we can define $\alpha(\cq):=\Ob(\kappa)$ where $\kappa : K \rightarrow \Out(M)$ is any locally split kernel on an infinite factor $M$ such that $\cq_\kappa=\cq$. It is clearly a group morphism thanks to Proposition \ref{operation on kernels}. Finally, $\alpha$ is injective because if $\Ob(\kappa)=0$ then $\kappa$ splits, hence $\cq_\kappa=0$.
   \end{proof}

Now, we describe explicitely the map $\alpha$ by relating it to the isomorphism $\beta$ introduced in Section \ref{integral cohomology}.

\begin{theorem}
    Let $K$ be a compact connected abelian group. Then the natural maps 
     $$\alpha : \Qint(K) \rightarrow H^3(K,\T) \quad \text{and} \quad \beta : \Qint(K) \rightarrow H^4(K,\Z)$$
    are related by $\beta = \delta \circ \alpha$ where $\delta : H^3(K,\T) \rightarrow H^4(K,\Z)$ is the Bockstein isomorphism.
\end{theorem}
\begin{proof}
    Let $\cq \in \Qint(K)$ and let $\kappa : K \rightarrow \Out(M)$ be a locally split kernel on some factor $M$ such that $\cq_\kappa=\cq$.  Let $\cb \in \Bil(\Lie(K))$ be an integral bilinear form such that $\cb(\xi,\xi)=\cq(\xi)$ for all $\xi \in \Lie(K)$.  Let $ p : G \rightarrow K$ be a covering and $\rho : G \rightarrow \Aut(M)$ be a $(\kappa \circ p)$-splitting action such that $\cb_\rho=\cb$. By Proposition \ref{cocompact bicharacter}, we may assume that $\chi_\rho = \exp_\T \circ B|_{G \times \ker p}$ for some bilinear form $B \in \Bil(G)$. Since $\rd B=\cb$ is integral, then by Proposition \ref{integral covering of integral bilinear form} and Proposition \ref{local equivalence and tangent form}, we may assume that $B$ is integral. This means that $\chi_\rho|_{N \times N}$ is trivial where $N=\ker p$.
        
    Let $s : K \rightarrow G$ be a measurable section. Then $k \mapsto \rho(s(k))$ is a measurable lift of $\kappa : K \rightarrow \Out(M)$. For all $k,h \in K$, we have 
   $$ \rho(s(h)) \rho(s(k))= \rho( \partial s (h,k) ) \rho(s(hk)).$$
   Write $\rho(\partial s(h,k))=\Ad(u(h,k))$ for some measurable map $(h,k) \mapsto u(h,k) \in \cU(M)$. Then the obstruction $\Ob(\kappa)$ is given by the cohomology class of the 3-cocycle
   \begin{align*}
       c(g,h,k)&=\rho(s(g))(u(h,k)) u(g,hk)  u(gh,k)^{-1} u(g,h)^{-1} \\
        &= \chi_\rho(s(g),\partial s(h,k)) \cdot u(h,k)u(g,hk) u(gh,k)^{-1}  u(g,h)^{-1}
   \end{align*}
    To simplify this formula, we will choose the measurable map $(h,k) \mapsto u(h,k)$ carefully.

    Choose unitaries $(v_n)_{n \in N}$ such that $\rho(n)=\Ad(v_n)$ for all $n \in N$. Since $\chi_\rho$ is trivial on $N \times N$, all these unitaries $(v_n)_{n \in N}$ must commute, hence they all live in the same abelian subalgebra $A \subset M$. We can view $$(h,k) \mapsto \rho( \partial s (h,k) )=\Ad(v_{\partial s(h,k)})$$
   as a symmetric $2$-cocycle with values in $\cU(A)/\T$. By Lemma \ref{symmetric vanishing cocycle}, it must be a coboundary, i.e.\ there exists a measurable map $w : K \rightarrow \cU(A)$ such that $$\rho( \partial s (h,k) ) = \Ad(w(h)w(k)w(hk)^{-1}).$$
   In other words, we can choose $u(h,k)=w(h)w(k)w(hk)^{-1}$ for all $h,k \in K$. This means that $u$ is a $2$-coboundary and a fortiori a 2-cocycle, hence we have 
   $$ u(h,k)u(g,hk) u(gh,k)^{-1} u(g,h)^{-1}=1 $$
   for all $g,h,k \in K$. With this choice, we can now simplify the formula for $c$ to obtain 
   $$ c(g,h,k)=\chi_\rho(s(g),\partial s (h,k))$$
   for all $g,h,k \in K$.

   Now, define $C \in C^3(K,\R)$ by the formula 
   $$C = B(s \wedge \partial s) : (g,h,k) \mapsto B(s(g),\partial s(h,k)).$$ 
   Since $c=\exp_\T ( C)$, then, by definition of $\delta$, we have 
   $$\delta([c])=[\partial C]= [B(\partial s \wedge \partial s)]=B(c_p \wedge c_p)=\beta(\cq).$$
    
   By definition of $\alpha$, we have $\alpha(\cq)=\Ob(\kappa)=[c]$. We conclude that $\delta( \alpha(\cq))=\beta(\cq)$. 
\end{proof}

With Theorem \ref{beta isomorphism} and Theorem \ref{quadratic from kernel}, we obtain the following corollaries.
\begin{corollary} \label{alpha isomorphism}
     Let $K$ be a compact connected abelian group. The map $$\alpha : \Qint(K) \rightarrow H^3(K,\T)$$
    is an isomorphism. 
\end{corollary} 

 \begin{corollary}
     Let $K$ be a compact connected abelian group. Then for every $\theta \in H^3(K,\T)$, there exists a locally split kernel $\kappa : K \rightarrow \Out(M)$ on a $\II_1$ factor $M$, such that $\Ob(\kappa)=\theta$.

     Moreover, we can take $M$ to be hyperfinite or full.
 \end{corollary}

 \begin{corollary} \label{all are locally split}
     Let $K$ be a compact connected abelian group. For every $\theta \in H^3(K,\T)$, there exists a covering $p : G \rightarrow K$ such that $p^*(\theta)$ vanishes in $H^3(G,\T)$.
     
     Consequently, every kernel $\kappa : K \rightarrow \Out(M)$ on an infinite factor $M$ is locally split.
 \end{corollary}
  \begin{proof}
     Take $\theta \in H^3(K,\T)$. There exists a locally split kernel $\kappa : K \rightarrow \Out(M)$ on some infinite factor $M$ with $\Ob(\kappa)=\theta$. Let $p : G \rightarrow K$ be a covering group such that $\kappa \circ p$ splits. Then $\Ob(\kappa \circ p)=p^*(\theta)$ must vanish.
     
     Now, take $\kappa : K \rightarrow \Out(M)$ an arbitrary kernel on an infinite factor $M$ with $\Ob(\kappa)=\theta$. Then $\Ob(\kappa \circ p)=p^*(\theta)=0$. Thus $\kappa \circ p$ splits \cite[Theorem 4.1.3]{Su80}.
 \end{proof}

Thanks to this corollary, for every kernel $\kappa : K \rightarrow \Out(M)$ on a factor $M$, the kernel $\kappa^\infty=\kappa \otimes \id$ on $M \ovt F$ is locally split where $F$ is an infinite type I factor. Thus, we can define $\cq_\kappa \: \dot{=} \cq_{\kappa^\infty}$. This definition does not depend on the choice of $F$ and satisfies the same properties as in Proposition \ref{operation on kernels} and Proposition \ref{alpha quadratic to obstruction}.

\begin{remark}
Using Theorem \ref{integral cohomology ring}, one can prove Corollary  \ref{all are locally split} directly without using kernels on von Neumann algebras. In fact, one can prove the following statement : for every compact connected abelian group $K$, every $n \geq 2$ and every $\theta \in H^n(K,\T)$, there exists a covering $p : G \rightarrow K$ such that $p^*(\theta)=0$ in $H^n(G,\T)$.
\end{remark}

\section{Factors with almost periodic outer modular flow}
Let $M$ be a von Neumann algebre . We denote by $\cP(M)$ the set of all faithful normal semifinite weights on $M$. We denote by $\sigma^M : \R \rightarrow \Out(M)$ the outer modular flow of $M$ defined by $\sigma^M=\epsilon_M \circ \sigma^\varphi$ for any $\varphi \in \cP(M)$ (see \cite{Co72}).
 
 \begin{theorem} \label{modular isotropic vector}
    Let $K$ be a compact abelian group with a one-parameter group $\iota \in  \Hom(\R,K)$ generated by $\xi \in \Lie(K)$. If $\kappa : K \rightarrow \Out(M)$ is a kernel on a factor $M$ such that $\kappa \circ \iota=\sigma^M$, then $\cq_\kappa(\xi)=0$.
    
    Conversely, if $\cq \in \Qint(K)$ satisfies $\cq(\xi)=0$, then there exists a full factor $M$ and a kernel $\kappa : K \rightarrow \Out(M)$ such that $\cq_\kappa=\cq$ and $\kappa \circ \iota=\sigma^M$.
\end{theorem}
\begin{proof}
    (1) Let $p : G \rightarrow K$ be a covering and $\rho : G \rightarrow \Aut(M)$ a $(\kappa \circ p)$-splitting action. Thanks to Proposition \ref{cocompact bicharacter}, we may assume that $\chi_\rho = \exp_\T \circ B|_{G \times \ker p}$ for some $B \in \Bil(G)$. Then $\cq_\kappa(\eta)=\rd B(\eta,\eta)$ for all $\eta \in \Lie(K)=\Lie(G)$. 

    Let $j : \R \rightarrow G$ be the one-parameter subgroup generated by $\xi \in \Lie(G)$ so that $p \circ j=\iota$. Then we have
     $\epsilon_M \circ \rho \circ j=\kappa \circ \iota = \sigma^M$.
    Therefore $\rho \circ j=\sigma^\varphi$ for some normal faithful semifinite weight $\varphi \in \cP(M)$. Since $\sigma^\varphi$ commutes with $\rho$, there exists $\omega \in \Bic(\R,G)$ such that $\rho(g)(\varphi^{\ri t})= \omega(t,g) \varphi^{\ri t}$ for all $(t,g) \in \R \times G$. We can write $\omega(t,g)=\exp_\T( t \phi(g))$ for some $\phi \in \Hom(G,\R)$ and all $(t,g) \in \R \times G$.
    
    Now, take $n \in \ker p$ and $u \in \cU(M)$ such that $\rho(n)=\Ad(u)$. For all $t \in \R$, we have 
    $$\omega(-t,n)=\varphi^{\ri t} \Ad(u)(\varphi^{-\ri t}) = \sigma_{t}^\varphi(u)u^*=\rho(j(t))(u)u^*= \chi_\rho(j(t),n) = \exp_\T( t B(j(1),n)).$$
    This means that $\phi(n)=-B(j(1),n)$ for all $n \in \ker p$. Since $\ker p$ is cocompact in $G$, this implies that $\phi(g)=-B(j(1),g)$ for all $g \in G$.

    Now, observe that $\omega(t,j(t))=1$ for all $t \in \R$. This means that $\phi(j(1))=0$, hence $B(j(1),j(1))=0$. Since $j(1)=\exp_G(\xi)$, we conclude that $$\mathfrak{q}_\kappa(\xi)=\rd B(\xi,\xi)=B(j(1),j(1))=0.$$

    (2) By Theorem \ref{kernel from quadratic}, we can find a $\II_1$ factor $P$ and a kernel $\kappa : K \rightarrow \Out(P)$ with $\cq_\kappa=\cq$. Let $p: G \rightarrow K$ be a covering and $\rho : G \rightarrow \Aut(P)$ a $\kappa \circ p$-splitting action. Thanks to Proposition \ref{cocompact bicharacter}, we may assume that $\chi_\rho$ extends to some $\chi \in \Bic(G)$. Let $j : \R \rightarrow G$ be the one-parameter subgroup generated by $\xi$. Since $\rd \chi(\xi,\xi)=\cq_\kappa(\xi)=0$, then $\chi \circ (j \times j)$ is trivial. 
    
    We now use an idea inspired by the construction of \cite{Hou07} to create a weight $\psi$ whose modular flow looks like $\rho \circ j$. Let $Q$ be a type $\III_1$ factor. Choose a faithful semifinite weight $\varphi \in \cP(Q)$. Define a flow $\gamma : \R \curvearrowright P \ovt Q$ by $\gamma_t=\rho(j(t)) \otimes \sigma_{-t}^\varphi$ for all $t \in \R$. Consider the crossed product $M=(P \ovt Q) \rtimes_{\gamma} \R$. Since $Q$ is of type $\III_1$, then $\sigma^\varphi$ is strictly outer by Connes and Takesaki relative commutant theorem \cite{CT76}. Therefore $\gamma$ is also strictly outer \cite[Proposition 10]{MV23} and $M$ is a factor. 
    
    Let $\phi \in \cP(M)$ be the dual weight of $\tau \otimes \varphi$. Then $\sigma^\phi$ is the unique extension of $\sigma^\varphi$ to $M$ that fixes $P \rtimes \R$. Let $\pi : G \rightarrow \Aut(M)$ be the unique extension of $\rho$ to $M$ that fixes $Q$ and satisfies $$ \forall (t,g) \in \R \times G , \quad \pi(g)(u_t)=\overline{\chi(j(t),g)} u_t$$ where $(u_t)_{t \in \R}$ are the unitaries implementing the copy of $\R$ inside the crossed product $M=(P \ovt Q) \rtimes_\gamma \R$. Since $\chi \circ (j \times j)$ is trivial, we know that $\pi \circ j$ fixes the unitaries $(u_t)_{t \in \R}$. Therefore $\pi \circ j$ is the unique extension of $\rho \circ j$ to $M$ that fixes $Q \rtimes \R$, hence $(\pi \circ j)(t) = \Ad(u_{t}) \circ \sigma_{t}^\phi$ for all $t \in \R$. This means that $\pi \circ j=\sigma^\psi$ for some faithful semifinite $\psi \in \cP(M)$.

    We claim that if $n \in \ker(p)$ and $\rho(n)=\Ad(v)$ for some $v \in \cU(P)$, then $\pi(n)=\Ad(v)$ on $M$. We already know that $\pi(n)$ and $\Ad(v)$ coincide on $P \ovt Q$, so to prove the claim, we just need to check that $\pi(n)(u_t)=v u_t v^*$ for all $t \in \R$. But, by definition, we have $\pi(n)(u_t)=\overline{\chi(j(t),n)}u_t$, while $$vu_tv^*=v \gamma_t(v^*)u_t=v \rho(j(t))(v^*) u_t=\chi(j(t),-n)u_t=\overline{\chi(j(t),n)}u_t.$$
    This proves the claim. In particuler, this shows that
    $(p,\pi)$ is a $\kappa'$-covering action for some kernel $\kappa' : K \rightarrow \Out(M)$. Moreover, we have $\chi_\pi=\chi_\rho$ because
    $$\chi_\pi(g,n)=\pi(g)(v)v^*= \rho(g)(v)v^*=\chi_\rho(g,n).$$
    We conclude that $\cq_{\kappa'}=\cq_\kappa=\cq$ and we check that $$\kappa' \circ \iota = \kappa' \circ p \circ j= \epsilon_M \circ \pi \circ j= \epsilon_M \circ \sigma^\psi=\sigma^M.$$

    Finally, observe that if $P$ and $Q$ are full and we choose $Q$ such that $\sigma^Q : \R \rightarrow \Out(Q)$ is a homeomorphism on its range, then $P \ovt Q$ will also be full full and $\epsilon_{P \ovt Q} \circ \gamma : \R \rightarrow \Out(P \ovt Q)$ will also be a homeomorphism on its range \cite{HMV16}. Therefore, $M$ will also be full \cite{MV23}. Such a factor $Q$ can be obtained as a free Araki-Woods factor associated to a mixing representation of $\R$. Note that the condition that $\sigma^Q : \R \rightarrow \Out(Q)$ is a homeomorphism on its range is equivalent to the property that $c(Q)$ is full \cite{Ma16}. 
\end{proof}

 Let $\Gamma < \R^*_+$ be a countable subgroup. We denote by $\iota_\Gamma : \R \rightarrow \widehat{\Gamma}$ the one-parameter subgroup obtained by duality from $\Gamma$. Following \cite{Co74}, we say that $\varphi \in \cP(M)$ is \emph{$\Gamma$-almost periodic} if there exists $\rho : \widehat{\Gamma} \rightarrow \Aut(M)$ such that $\rho \circ \iota_\Gamma=\sigma^\varphi$. In that case, $\rho$ is unique and we denote it by $\sigma^{\varphi,\Gamma}$. A weight $\varphi \in \cP(M)$ is almost periodic, i.e.\ $\sigma^\varphi(\R)$ has compact closure in $\Aut(M)$, if and only if $\varphi$ is $\Gamma$-almost periodic for some $\Gamma < \R^*_+$.
 
 Suppose that $M$ is a full factor. We say that $\sigma^M$ is $\Gamma$-almost periodic if there exists a continuous morphism $\kappa : \widehat{\Gamma} \rightarrow \Out(M)$ such that $\kappa \circ \iota_\Gamma=\sigma^M$. Then $\kappa$ is unique and we denote it by $\sigma^{M,\Gamma}$. There exists a smallest countable subgroup $\Gamma < \R^*_+$ such that $\sigma^M$ is $\Gamma$-almost periodic. It is the unique subgroup $\Gamma < \R^*_+$ such that $\sigma^{M,\Gamma}$ is injective. We denote this subgroup $\Gamma$ by $\Sd(M)$ because it coincides with the group $\Sd(M)$ defined in \cite{Co74} when $M$ has an almost periodic weight.

 \begin{proposition} \label{quadratic interpretation}
 Let $\Gamma < \R^*_+$ be a countable subgroup. 
 \begin{enumerate}
 \item The map
 $$ \log^2 : \Qint(\widehat{\Gamma}) \ni \rd \lambda \odot \rd \mu \mapsto \log(\lambda) \odot \log(\mu) \in \log(\Gamma) \odot \log(\Gamma)$$
 is a well-defined group isomorphism.
 \item If $\xi \in \Lie(\widehat{\Gamma})$ is the generator of $\iota_\Gamma : \R \rightarrow \widehat{\Gamma}$, then for every $\cq \in \Qint(\widehat{\Gamma})$, we have $$\cq(\xi)=\frac{1}{4\pi^2}m(\log^2 \cq)$$ where  $$m : \R \odot \R \ni x \odot y \mapsto xy  \in \R$$
 is the multiplication map.
 \end{enumerate}
 \end{proposition}
 \begin{proof}
 (1) Use Proposition \ref{isomorphism square quadratic}.
 
 (2) Take $\lambda \in \Gamma$. By definition $$\langle \lambda, \exp_{\widehat{\Gamma}}(t \xi) \rangle = \langle \lambda, \iota_\Gamma(t) \rangle=\lambda^{\ri t}=e^{\ri t \log(\lambda)} = \exp_\T ( t \log(\lambda) /2\pi)$$ for all $t \in \R$. This means that the linear form $\rd \lambda \in \Lie(\widehat{\Gamma})^*$ satisfies $\rd \lambda(\xi)=\log(\lambda)/2\pi$. Therefore, we have
 $$ (\rd \lambda \odot \rd \mu)(\xi)=\rd \lambda(\xi) \rd \mu(\xi)=\log(\lambda)\log(\mu)/4\pi^2$$
 and this is exactly what we wanted.
 \end{proof}
 
  \begin{definition}
 We say that $r \in \R \odot \R$ is a \emph{quadratic relation} if it belongs to the kernel of the multiplication map $m :\R \odot \R \rightarrow \R$.
 
 We say that a subgroup $T < \R$ is \emph{quadratically free} if $T \odot T$ contains no nontrivial quadratic relations or equivalently, if the multiplication map $m : T \odot T \rightarrow \R$ is injective.
 \end{definition}
 \begin{remark}
Clearly, a subgroup $T < \R$ is quadratically free if and only if every finitely generated subgroup of $T$ is quadratically free. Moreover, if $(t_1,\dots,t_n) \in \R^n$ is a $\Z$-linearly independent family of real numbers, then the subgroup $T=\langle t_1, \dots, t_n\rangle$ is quadratically free if and only if the family $(t_it_j)_{i \leq j}$ is $\Z$-linearly independent.
 \end{remark}
 
 We can now reformulate Theorem \ref{modular isotropic vector}.
 
 \begin{theorem}
 Let $M$ be a full factor such that $\sigma^M$ is almost periodic. Let $\Gamma=\Sd(M)$. Then
 $$\varrho(M) \dot{=} \log^2 \cq_{\sigma^{M,\Gamma}} \in \log(\Gamma) \odot \log(\Gamma)$$
 is a quadratic relation and $\varrho(M)=0$ if and only if $M$ has an almost periodic weight.

Conversely, if $\Gamma < \R^*_+$ is a countable group and $r \in \log(\Gamma) \odot \log(\Gamma)$ is a quadratic relation, there exists a full factor $M$ such that $\sigma^M$ is almost periodic, $\Sd(M)=\Gamma$ and $\varrho(M)=r$.
 \end{theorem}
 
 We can now deduce our main result.
 
  \begin{corollary}
 Let $\Gamma < \R^*_+$ be a countable group. Then $\log(\Gamma)$ is quadratically free if and only if every full factor $M$ such that $\sigma^M$ is $\Gamma$-almost periodic has an almost periodic weight.
 \end{corollary}

We will now study how obstructions behave for tensor products. For this we need the following general lemma.

\begin{lemma} \label{kernel factorize}
    Let $M$ and $N$ be two factors, $G$ a locally compact group and $\kappa : G \rightarrow \Out(M \ovt N)$ a kernel. Suppose that $M$ is full. Then there exists an open subgroup $G_0 < G$ and two kernels $\gamma : G_0 \rightarrow \Out(M)$ and $\beta : G_0 \rightarrow \Out(N)$ such that $\kappa|_{G_0}=\gamma \otimes \delta$.
\end{lemma}
\begin{proof}
By \cite[Theorem 6.1]{IM22}, we know that the continuous map
$$ \varphi : \cU(M \ovt N) \times \Aut(M) \times \Aut(N) \rightarrow \Aut(M \ovt N)$$
$$ (u, \alpha,\beta) \mapsto \Ad(u) \circ (\alpha \otimes \beta)$$
is open. Let $\theta : G \rightarrow \Aut(M \ovt N)$ be a borel lift of $\kappa$.  Consider the quotient group morphism $$p : \Aut(M \ovt N) \rightarrow \underline{\Out}(M \ovt N).$$ Since $\underline{\Out}(M \ovt N)$ is a Hausdorff Polish group and $p \circ \theta$ is a borel group morphism, we know that $p \circ \theta$ is automatically continuous. Let $U$ be the image of $\varphi$, which is an open subgroup of $\Aut(M\ovt N)$. Then $p(U)$ is an open subgroup of $\underline{\Out}(M \ovt N)$. Thus $G_0:=(p \circ \theta)^{-1}(p(U))$ is an open subgroup of $G$. Since $\ker p=\overline{\Inn(M \ovt N)}  \subset U$, we have $p^{-1}(p(U))=U$, hence $G_0=\theta^{-1}(U)$. In particular, $\theta(G_0)$ is contained $U$, which is the image of $\varphi$. Thus, we can find a borel lift 
$$\tilde{\theta} : G_0 \rightarrow \cU(M \ovt N) \times \Aut(M) \times \Aut(N) $$
such that $\varphi \circ \tilde{\theta} = \theta|_{G_0}$.
Write $\tilde{\theta}_g = (u_g, \alpha_g,\beta_g)$ for all $g \in G_0$. It is straightforward to check that $\gamma : g \mapsto \epsilon_M(\alpha_g) \in \Out(M)$ and $\delta : g \mapsto \epsilon_N(\beta_g) \in \Out(N)$ are two kernels and that $\kappa_g = \gamma_g \otimes \delta_g$ for all $g \in G_0$, as we wanted.
\end{proof}

\begin{proposition}
    Let $M$ and $N$ be two full factors. Then $\sigma^{M \ovt N}$ is almost periodic if and only if $\sigma^M$ and $\sigma^N$ are both almost periodic. In that case, we have $\Sd(M \ovt N) = \Sd(M) \cdot \Sd(N)$ and $\varrho(M \ovt N)=\varrho(M) + \varrho(N)$.
\end{proposition}
\begin{proof}
The proposition follows from Lemma \ref{kernel factorize} and Proposition \ref{operation on kernels}. We leave the details to the reader.
\end{proof}

\begin{proposition}
    Let $M$ be a full factor such that $\sigma^M$ is almost periodic. Then $\Sd(M^{\rm op}) = \Sd(M)$ and $\varrho(M^{\rm op}) = \varrho(M)$.
\end{proposition}
\begin{proof}
Let $\Gamma= \Sd(M)$ and $\kappa=\sigma^{M,\Gamma} : \widehat{\Gamma} \rightarrow \Out(M)$. Let $\theta \in \Aut(\widehat{\Gamma})$ be defined by $\theta(g)=g^{-1}$ for all $g \in \widehat{\Gamma}$. We have $\sigma^{M^{\op}}_t = (\sigma_{-t}^M )^{\op}$ for all $t \in \R$. Therefore $\kappa^{\op} \circ \theta$ is a $\Gamma$-compactification of $\sigma^{M^{\op}}$. This shows that $\Sd(M^{\rm op}) < \Gamma$ and by symmetry, $\Gamma=\Sd(M^{\rm op})$.
    
 Moreover, by Proposition \ref{operation on kernels}, we have
    $$ \cq_{M^{\op}}=\cq_{\kappa^{\op}} \circ \rd \theta = (-\cq_\kappa) \circ (-\id)=\cq_\kappa=\cq_M.$$
\end{proof}

Let $M$ be an arbitrary factor (not necessarily full). A \emph{$\Gamma$-compactification} of $\sigma^M$ is a kernel $\kappa : \widehat{\Gamma} \rightarrow \Out(M)$ such that $\kappa \circ \iota_\Gamma = \sigma^M$ where $\iota_\Gamma : \R \rightarrow \widehat{\Gamma}$ is the dense one-parameter subgroup obtained by duality from $\Gamma < \R^*_+$. When such a $\Gamma$-compactification exists, we say that $\sigma^M$ is \emph{$\Gamma$-almost periodic}. We say that $\sigma^M$ is almost periodic if it is $\Gamma$-almost periodic for some $\Gamma < \R^*_+$. Note that if $M$ is not full, then a $\Gamma$-compactification of $\sigma^M$, if it exists, is not necessarily unique.

Suppose that $\kappa$ is a  $\Gamma$-compactification of $\sigma^M$ for some countable subgroup $\Gamma < \R^*_+$. If $\kappa$ is split, then every $\kappa$-splitting action is of the form $\sigma^{\varphi,\Gamma}$ for some $\Gamma$-almost periodic weight $\varphi \in \cP(M)$.

\begin{proposition}
    Let $M$ and $N$ be two factors and suppose that $M$ is full. Then $\sigma^{M \ovt N}$ is almost periodic if and only if $\sigma^M$ and $\sigma^N$ are both almost periodic.

    Moreover, if $M \ovt N$ and $M$ are both almost periodic, then $N$ is also almost periodic.
\end{proposition}
\begin{proof}
    Suppose that $\kappa$ is a $\Gamma$-compactification of $\sigma^{M \ovt N}$ for some $\Gamma < \R^*_+$. Then by Lemma \ref{kernel factorize} we can write $\kappa=\gamma \otimes \delta$ where $\gamma : \widehat{\Gamma} \rightarrow \Out(M)$ and $\delta  : \widehat{\Gamma} \rightarrow \Out(N)$ are two kernels. Clearly, they are $\Gamma$-compactifications of $\sigma^M$ and $\sigma^N$ respectively. 

    Suppose that $M \ovt N$ is almost periodic. Then we can choose $\kappa$ such that $\cq_\kappa=0$. Suppose moreover that $M$ is almost periodic. Since $M$ is full, then $\cq_\gamma=\cq_M=0$. We conclude that $\cq_\delta=\cq_\kappa-\cq_\gamma=0$, which means that $N$ is almost periodic.
\end{proof}

\begin{theorem} \label{trivial isotropic vector}
    Let $K$ be a compact abelian group with a one-parameter group $\iota \in  \Hom(\R,K)$ generated by $\xi \in \Lie(K)$. If $\kappa : K \rightarrow \Out(M)$ is a kernel on a factor $M$ such that $\kappa \circ \iota$ is trivial, then $\cq_\kappa(\xi)=0$.
    
    Conversely, if $\cq \in \Qint(K)$ satisfies $\cq(\xi)=0$, then there exists a kernel $\kappa : K \rightarrow \Out(M)$ on a type $\II$ factor $M$ such that $\cq_\kappa=\cq$ and $\kappa \circ \iota$ is trivial. Moreover, we can take $M$ to be amenable.
\end{theorem}
\begin{proof}
    (1) Let $p : G \rightarrow K$ be a covering and $\rho : G \rightarrow \Aut(M)$ a $(\kappa \circ p)$-splitting action. Thanks to Proposition \ref{cocompact bicharacter}, we may assume that $\chi_\rho = \exp_\T \circ B|_{G \times \ker p}$ for some $B \in \Bil(G)$. Then $\cq_\kappa(\eta)=\rd B(\eta,\eta)$ for all $\eta \in \Lie(K)=\Lie(G)$. 

    Let $j : \R \rightarrow G$ be the one-parameter subgroup generated by $\xi \in \Lie(G)$ so that $p \circ j=\iota$. Then
     $\epsilon_M \circ \rho \circ j=\kappa \circ \iota$ is trivial. Therefore $\rho \circ j$ is an inner flow : $\rho(j(t))=\Ad(v_t)$ for all $t \in \R$, where $t \mapsto v_t$ is one-parameter unitary group in $M$. Since $\Ad(v_t)$ commutes with $\rho$, there exists $\omega \in \Bic(\R,G)$ such that $\rho(g)(v_t)= \omega(t,g) v_t$ for all $(t,g) \in \R \times G$. We can write $\omega(t,g)=\exp_\T( t \phi(g))$ for some $\phi \in \Hom(G,\R)$ and all $(t,g) \in \R \times G$.
    
    Now, take $n \in \ker p$ and $u \in \cU(M)$ such that $\rho(n)=\Ad(u)$. For all $t \in \R$, we have 
    $$\omega(-t,n)=v_t \Ad(u)(v_t^*) = \Ad(v_t)(u)u^*=\rho(j(t))(u)u^*= \chi_\rho(j(t),n) = \exp_\T( t B(j(1),n)).$$
    This means that $\phi(n)=-B(j(1),n)$ for all $n \in \ker p$. Since $\ker p$ is cocompact in $G$, this implies that $\phi(g)=-B(j(1),g)$ for all $g \in G$.

    Now, observe that $\omega(t,j(t))=1$ for all $t \in \R$. This means that $\phi(j(1))=0$, hence $B(j(1),j(1))=0$. Since $j(1)=\exp_G(\xi)$, we conclude that $$\mathfrak{q}_\kappa(\xi)=\rd B(\xi,\xi)=B(j(1),j(1))=0.$$

    (2) By Theorem \ref{kernel from quadratic}, we can find a $\II_1$ factor $P$ and a kernel $\kappa : K \rightarrow \Out(P)$ with $\cq_\kappa=\cq$. 
    %Up to replacing $\kappa$ by $\kappa \otimes \kappa_0$ where $\kappa_0 : K \rightarrow \Out(P_0)$ is a split kernel such that $\kappa_0 \circ \iota$ is strictly outer, we may assume that $\kappa \circ \iota$ is strictly outer.
    Let $p: G \rightarrow K$ be a covering and $\rho : G \rightarrow \Aut(P)$ a $\kappa \circ p$-splitting action. Thanks to Proposition \ref{cocompact bicharacter}, we may assume that $\chi_\rho$ extends to some $\chi \in \Bic(G)$. Let $j : \R \rightarrow G$ be the one-parameter subgroup generated by $\xi$. Since $\rd \chi(\xi,\xi)=\cq_\kappa(\xi)=0$, then $\chi \circ (j \times j)$ is trivial. 
    
    Let $Q$ be a type $\II_1$ factor and $\alpha : K \curvearrowright Q$ a strictly outer action such that $\alpha \circ \iota$ is also strictly outer (use for example infinite tensor product of an outer action \cite{Va05}). Define a flow $\gamma : \R \curvearrowright P \ovt Q$ by $\gamma_t = \rho(j(t)) \otimes \alpha(\iota(t))$ for all $t \in \R$.
    Consider the crossed product $M=(P \ovt Q) \rtimes_{\gamma} \R$. Since $\alpha \circ \iota$ is strictly outer, $\gamma$ is also strictly outer and $M$ is a factor \cite[Proposition 10]{MV23}. 
    
    Consider the diagonal action $\rho \otimes (\alpha \circ p) : G \curvearrowright P \ovt Q$. Let $\pi : G \rightarrow \Aut(M)$ be the unique extension of this diagonal action to $M$ that satisfies $$ \forall (t,g) \in \R \times G , \quad \pi(g)(u_t)=\overline{\chi(j(t),g)} u_t$$ where $(u_t)_{t \in \R}$ are the unitaries implementing the copy of $\R$ inside the crossed product $M=(P \ovt Q) \rtimes_\gamma \R$. Since $\chi \circ (j \times j)$ is trivial, we know that $\pi \circ j$ fixes the unitaries $(u_t)_{t \in \R}$. Therefore $(\pi \circ j)(t) = \Ad(u_{t})$ for all $t \in \R$. This means that $\epsilon_M \circ \pi \circ j$ is trivial.

    We claim that if $n \in \ker(p)$ and $\rho(n)=\Ad(v)$ for some $v \in \cU(P)$, then $\pi(n)=\Ad(v)$ on $M$. We already know that $\pi(n)$ and $\Ad(v)$ coincide on $P \ovt Q$, so to prove the claim, we just need to check that $\pi(n)(u_t)=v u_t v^*$ for all $t \in \R$. But, by definition, we have $\pi(n)(u_t)=\overline{\chi(j(t),n)}u_t$, while $$vu_tv^*=v \gamma_t(v^*)u_t=v \rho(j(t))(v^*) u_t=\chi(j(t),-n)u_t=\overline{\chi(j(t),n)}u_t.$$
    This proves the claim. In particuler, this shows that
    $(p,\pi)$ is a $\kappa'$-covering action of some kernel $\kappa' : K \rightarrow \Out(M)$. Moreover, we have $\chi_\pi=\chi_\rho$ because
    $$\chi_\pi(g,n)=\pi(g)(v)v^*= \rho(g)(v)v^*=\chi_\rho(g,n).$$
    We conclude that $\cq_{\kappa'}=\cq_\kappa=\cq$ and by construction, $\kappa' \circ \iota = \kappa' \circ p \circ j= \epsilon_M \circ \pi \circ j$ is trivial.

    Finally, observe that if $P$ and $Q$ are amenable, then $M$ is also amenable.
    %Observe that $$c(M) \cong M \rtimes_{\sigma^\phi} \R \cong (P \rtimes_{\alpha \circ \iota'} \R) \ovt (N \rtimes_{\sigma^\varphi} \R).$$
    %Since $\alpha \circ \iota'$ and $\sigma^\varphi$ are both outer, we see that $c(M)$ is factor (but it is not full). Thus $M$ is of type $\III_1$.
\end{proof}

\begin{corollary}
    Let $M$ be a factor. If $\sigma^M$ is almost periodic then $M \ovt R$ is almost periodic, where $R$ is the hyperfinite $\II_1$ factor.
\end{corollary}
\begin{proof}
    Let $K$ be a compact abelian group, $\kappa : K \rightarrow \Out(M)$ a kernel and $\iota \in \Hom(\R,K)$ a one-parameter subgroup such that $\sigma^M=\kappa \circ \iota$. By Theorem \ref{modular isotropic vector}, there exists an amenable type $\II$ factor $P$ and a kernel $\kappa' : K \rightarrow \Out(P)$ such that $\kappa' \circ \iota$ is trivial and $\cq_{\kappa'}=-\cq_{\kappa}$. Consider the kernel $\kappa \otimes \kappa' : K \rightarrow \Out(M \ovt P)$. Then, we have $(\kappa \otimes \kappa') \circ \iota = \sigma^M \otimes \id= \sigma^{M \ovt P}$. Moreover, we have $\cq_{\kappa \otimes \kappa'}=\cq_\kappa + \cq_{\kappa'}=0$. Thus $\kappa \otimes \kappa'$ splits, which means that $M \ovt P$ has an almost periodic weight.
\end{proof}

Combining all the previous results of this section, we obtain the following theorem.
\begin{theorem}
    Let $M$ be a factor. Consider the following properties.
    \begin{enumerate}
        \item $\sigma^M$ is almost periodic.
        \item There exists a full factor $N$ such that $M \ovt N$ is almost periodic.
        \item $M \ovt R$ is almost periodic where $R$ is the hyperfinite $\II_1$ factor.
        \item There exists a factor $N$ such that $M \ovt N$ is almost periodic.
    \end{enumerate}
    Then $(1) \Leftrightarrow (2) \Rightarrow (3) \Rightarrow (4)$. If $M$ is full, then all these properties are equivalent.
\end{theorem}

\section{Equivalence relations}
We refer to \cite{HMV17} for remindes on the cohomology of equivalence relations and the theory of almost periodic measures for nonsingular equivalence relations.

\begin{proposition}
 Let $\cR$ be a nonsingular equivalence relation on a provability space $(X,\mu)$ and $G$ a locally compact abelian group. Let $\kappa : G \rightarrow H^1(\cR,\T)$ be a kernel, i.e.\ a group morphism that admits a measurable lift to $Z^1(\cR,\T)$.
 
  Then $\kappa$ is split and there exists $\chi \in Z^1(\cR, \widehat{G})$ such that $\kappa(g)=[\langle g, \chi \rangle ]$ for all $g \in G$.
\end{proposition}
\begin{proof}
    Let $\rho : G \rightarrow Z^1(\cR,\T)$ be a measurable lift of $\kappa$. Then $c : (g,h) \mapsto \rho(g)\rho(h)\rho(gh)^{-1}$ is a symmetric  $2$-cocycle with values in $B^1(\cR,\T)$. Since $B^1(\cR,\T)=\cU(A)/\T$ where $A=\rL^\infty(X)$, we conclude by Lemma \ref{symmetric vanishing cocycle}, that $c$ is a coboundary, i.e.\ there exists a measurable map $b : g \mapsto B^1(\cR,\T)$ such that $c(g,h)=b(g)b(h)b(gh)^{-1}$ for all $g,h \in G$. Therefore, if we define $\rho' : G \rightarrow Z^1(\cR,\T)$ by $\rho'(g)=b(g)^{-1} \rho(g)$ for all $g \in G$, then $\rho'$ is a measurable morphism from $G$ into the polish group $Z^1(\cR,\T)$, hence it is automatically continuous. This shows that $\kappa$ is split. 
    
    Now, for $(x,y) \in \cR$, the map $g \mapsto \rho'(g)(x,y)$ is a character on $G$, hence we can write $\rho'(g)(x,y)=\langle g, \\chi(x,y) \rangle$ for some $\chi(x,y) \in \widehat{G}$. It is straightforward, to check the measurability of the map $\chi : \cR \ni (x,y) \mapsto \chi(x,y) \in \widehat{G}$ and that $\chi \in \Z^1(\cR,\widehat{G})$. By construction we have $\rho'(g)=\langle g, \chi\rangle$ as we wanted.
\end{proof}

As a special case of the previous proposition, we obtain the following positive answer to Question \ref{main question} in the framework of equivalence relation.
\begin{proposition} \label{prop equivalence}
Let $\cR$ be a strongly ergodic nonsingular equivalence relation on some probability space $(X,\mu)$. Let $\delta_\mu \in Z^1(\cR,\R^*_+)$ be the Radon-Nikodym cocycle and $\delta_\cR=[\delta_\mu] \in H^1(\cR,\R^*_+)$ be the canonical Radon-Nikodym cohomlogy class. Then $\cR$ admits a measure $\nu \sim \mu$ that is $\cR$-almost periodic ($\delta_\nu$ has countable range) if and only if the continuous one-parameter group $\sigma_\cR : t \mapsto \delta_\cR^{\ri t} \in H^1(\cR,\T)$ is almost periodic ($\overline{\sigma_\cR(\R)}$ is compact).
\end{proposition}

Finally, Theorem \ref{main equivalence relation} follows from the previous proposition combined with \cite[theorem E]{HMV17} which says that $\sigma_\cR : \R \rightarrow  H^1(\cR,\T)$ and $\sigma_M : \R \rightarrow \Out(M)$ define the same topology on $\R$ when $\cR$ is the orbit equivalence relation of a strongly ergodic action of a bi-exact group and $M$ is the associated crossed product.

\bibliographystyle{plain}

\end{document}